\documentclass[10pt]{article}
\usepackage[a4paper, left=3cm, right=3cm, height=22cm]{geometry}
\usepackage{amsmath,amsfonts,amsthm,bbm, amssymb}
\usepackage[arrow,matrix]{xy}
\usepackage{enumerate}

\usepackage[usenames]{color}
\usepackage{amssymb}
\usepackage{ulem}
\usepackage[inline]{showlabels}

\theoremstyle{plain}
\numberwithin{equation}{section}
\newtheorem{thm}{Theorem}[section]
\newtheorem{cor}[thm]{Corollary}
\newtheorem{lem}[thm]{Lemma}
\newtheorem{prop}[thm]{Proposition}
\newtheorem{conj}[thm]{Conjecture}

\theoremstyle{definition}
\newtheorem{df}[thm]{Definition}
\newtheorem{exe}[thm]{Example}
\newtheorem{rmk}[thm]{Remark}

\usepackage[explicit]{titlesec}

\titleformat{\subsubsection}
  [runin] % to have a run-in title, like \paragraph
  {\normalfont\normalsize\bfseries} % bold
  {\thesubsubsection.} % the numbering
  {0.5em} % horizontal spacing
	{#1} % the title

\newcommand{\mb}{\mathbb}
\newcommand{\mf}{\mathfrak}
\newcommand{\ml}{\mathcal}

\newcommand{\Spt}{\mathbf{S}\mathbf{p}\mathbf{t}} %%% symmetric spectra
\newcommand{\Sym}{\Sigma}
\newcommand{\Sm}{\mathbf{Sm}_k}   %%% smooth varieties

\newcommand{\Shk}{\mathcal{S}h^{tr}_{Nis}(k)}
\newcommand{\R}{\mathbf{R}\mathbf{e}\mathbf{p}_{GL_2}}
\newcommand{\ot}{\otimes^{tr}}

\newcommand{\CM}{\mathcal{CM}}
\newcommand{\KCM}{\mathcal{KCM}}
\newcommand{\Hom}{\underline{RHom}}
\newcommand{\V}{\mathbf{F}}

\newcommand{\PST}{\mathbf{PST}}
\newcommand{\Ztr}{\mb{Z}_{tr}}
\newcommand{\DM}{\mathbf{DM}}
\newcommand{\DMEM}{\mathbf{DMEM}}
\newcommand{\MEM}{\mathbf{MEM}}

\newcommand{\Ab}{\mathbf{Ab}}

\usepackage[pdfauthor={Jin Cao}, %
				pdftitle={Cdga}%
			dvips,colorlinks=true]{hyperref}
\hypersetup{
	bookmarksnumbered=true,
	linkcolor=black,
}
\newcounter{elno}

\begin{document}
\author{Jin Cao}
\title{Motives for an elliptic curve}

\maketitle

%%%%%abstract
\begin{abstract}
In this paper we describe the rigid tensor triangulated subcategory of Voevodsky's triangulated category of motives generated by the motive of an elliptic curve as a derived category of dg modules over a commutative differential graded algebra in the category of representations over some reductive group.
\end{abstract}

%MSC
%\noindent{Mathematics Subject Classification (2010)}: 

%keywords
%\noindent{Keywords}: 

%\setcounter{tocdepth}{1}
%\tableofcontents

\section{Introduction}
During 1980s, Beilinson and Deligne independently describe a conjectural abelian tensor category of mixed motives $\mathbf{MM}(k, \mb{Q})$ over a given base field $k$. The existence of an abelian category of mixed motives would have important consequences for our understanding of smooth varieties. The category $\mathbf{MM}(k, \mb{Q})$ has yet to be constructed.  Alternatively, Voevodsky, Levine and Hanamura have independently constructed a triangulated category of mixed motives over a field, modeled on the derived category of the conjectural abelian category of mixed motives. Notably, Voevodsky's triangulated category of mixed motives satisfies most properties predicted by Beilinson. Then one may ask whether there is a reasonable t-structure on Voevodsky's triangulated category of motives with rational coefficient $\DM_{gm}(k, \mb{Q})$, which gives the desired abelian category of mixed motives. The only known example is mixed Tate motives \cite{L} (short for MTMs), i.e. the category of motives generated by Tate objects. In fact, if the base field $k$ satisfies the Beilinson-Soul\'e vanishing conjecture, Levine \cite{L} shows that the triangulated category of MTMs has a t-structure. 
\

Later Bloch and Kriz \cite{BK} provide a different way of constructing an abelian category of MTMs. Roughly speaking, the conjectured abelian category of mixed Tate motives $\mathbf{MTM}$ is a Tannakian category, whose Tannakian fundamental group $\pi_1(\mathbf{MTM})$ is an extension of a prounipotent algebraic group $U$ by the multiplicative group $\mb{G}_m$. Bloch and Kriz's work gives a description about one candidate of the prounipotent group $U$. This group has an explicit description in term of ''cycle algebras", therefore Bloch and Kriz's MTMs is defined as the category of graded representations over $U$. Then a natural question is: 
\

\textit{Does Bloch and Kriz's construction coincide with Levine's construction if the base field satisfies the Beilinson-Soul\'e vanishing conjecture? Or what's the relation between Bloch and Kriz's construction and Voevodsky's construction?}
\

Combining Bloch and Kriz's construction with Kriz and May's general theory of Adams cdgas \cite{KM},  Spitzweck \cite{S} defines an equivalence between the dervied category of Adams dg-modules over BK's cycle algebra and the rigid subcategory of $\DM_{gm}(k, \mb{Q})$ generated by Tate objects. As a corollary, if the Beilinson-Soul\'e vanishing conjecture is true for the base fields, all mentioned constructions of the abelian category of MTMs are the same.
\

In this paper, we continue with the viewpoint of cycle algebras to understand the motives generated by an elliptic curve $E$ defined over a base field $k$ \textbf{with characteristic zero}. We handle the case that the elliptic curve is without complex multiplication and with complex multiplication separately.  Like Tannakian fundamental group of mixed Tate motives, the conjectured Tannakian fundamental group for motives generated by a non-CM (resp. CM \footnote{For the CM case, we only consider the complex multiplication is defined over $k$.}) elliptic curve is an extension of a pro-unipotent algebraic group by $GL_2$ (resp. the Weil restriction $\mathbf{Res}_{\mb{K}/\mb{Q}}\mb{G}_m$, where $\mb{K} = End(E) \otimes \mb{Q}$). The elliptic cycle algebra should lie in the category of representations over $GL_2$ (resp. $\mathbf{Res}_{\mb{K}/\mb{Q}}\mb{G}_m$). However, $\mathbf{Res}_{\mb{K}/\mb{Q}}\mb{G}_m$ is not absolutely irreducible, which causes a lot of difficulties. Our strategy for the CM case is extending the cycles algebra and representations over $\mb{K}$ rather than $\mb{Q}$ and using the isomorphism $\mathbf{Res}_{\mb{K}/\mb{Q}}\mb{G}_m \otimes \mb{K} \cong \mb{G}_{m, \mb{K}} \times \mb{G}_{m, \mb{K}}$. Studying the cdga object in the category of $GL_2$-representations  is one of the main motivation for us to develop a theory of cdgas over some reductive group in \cite{C}, which generalize Kriz and May's theory of Adams cdgas. Compared with mixed Tate motives, the left things are constructing a reasonable elliptic type cdga and further connecting with $\DM_{gm}(k, \mb{Q})$ (resp. $\DM_{gm}(k, \mb{K})$), which are the main contents of our paper.
\

Let us explain the rough idea of the construction of the elliptic cycle algebra for non-CM case. The CM case is similar. For the desired elliptic cycle algebra $\ml{E}_{ell}^*$, as a $GL_2$-representation, it will decompose as a direct sum of irreducible representations. Therefore we need to figure out coefficients for every irreducible $GL_2$-representation. Recall every irreducible $GL_2$ representation has the form $Sym^a \V \otimes det^{\otimes b}$ for $a \in \mb{Z}_{\geq 0}$ and $b \in \mb{Z}$, where $\V$ is the standard $GL_2$ representation. For $Sym^a \V \otimes det^{\otimes b}$, the cohomology of the coefficient complex should reflect the extension of the motives $Sym^a M_1(E) \otimes \mb{Q}(b)$ by $\mb{Q}$ (See Lemma \ref{fully-faithful}). Our computations in Section 5 implies the cohomology groups of our construction of $\ml{E}_{ell}^*$ have these properties. As for the coefficient complexes, we choose the Friendlander-Suslin construction (reviewed in Section \ref{FS}), which is a functorial improvement of Bloch's cycle complexes.
\

After constructing the elliptic cycle algebra $\ml{E}_{ell}^*$, we use the sheaf version  $\mf{E}_{ell}^*$ (Definition \ref{sheaf version}) of $\ml{E}_{ell}^*$ and symmetric motivic $T^{tr}$-spectra to define the functor from the derived category of dg-$\ml{E}_{ell}^*$ modules to Voevodsky's big category of motives. Restricting to compact objects, we get the desired equivalence, which is our main results -- Theorem \ref{equi} and Theorem \ref{equicm}.

There are other related constructions or understanding of motives for an elliptic curve.
\begin{itemize}
\item
Patashnick \cite{P} constructs a different cycle algebra for an elliptic curve $E$ without CM and defines one candidate for the abelian category of motives for $E$. Compared to his work, the advantage of our construction is its identification with a full subcategory of $\DM_{gm}(k, \mb{Q})$. Another difference between Patashnick's construction and ours is the use of Friedlander-Suslin complexes in our paper rather than Bloch's cycle complexes. 
\item
We also mention that, besides the approach of cycle algebras along the lines of work of Bloch, Kriz, May et al., Kimura and Terasoma in \cite{KT} develop a theory of relative DGAs and used their theory to define another candidate for an abelian category of mixed elliptic motives. 
\item
Iwanari \cite{I1} uses derived Tannaka duality to describe the stable $\infty$-category of motives generated by a Kimura finite Chow motives as a symmetric monoidal stable $\infty$-category of quasi-coherent complexes on a derived quotient stack. In particular, motives for an elliptic curve are Kimura finite. Based on Tannakian formalism, he \cite{I2} further describes the derived motivic Galois group of $\infty$-category of motives generated by an abelian variety over a number field with some condition. 
\end{itemize}

In outline, the content of the paper is as follow: In Section \ref{motives for an ell}, \ref{FS}, we briefly recall some basic facts about the motives of an elliptic curve and Friedlander-Suslin complexes as preparations. We give the detailed construction of the elliptic cycle algebra $\ml{E}_{ell}^*$ in Section \ref{elliptic cycle algebra} and show their properties in Section \ref{sec-comp}, \ref{sec-compcm}. In order to connect with $\DM_{gm}(k, \mb{Q})$, we formulate the sheaf version of the elliptic cycle algebras in Section \ref{cycle alg}. Then we can construct a functor from the derived category of dg $\ml{E}_{ell}^*$-modules to Voevodsky's big category of motives $\mathbf{DM}(k, \mb{Q})$, whose restriction on the compact objects leads to a functor to $\DM_{gm}(k, \mb{Q})$. In Section \ref{DG modules and motives for an elliptic curve}, we provide such construction and show that this functor induces an equivalence between the compact objects in the derived category of dg $\ml{E}_{ell}^*$-modules and the idempotent complete rigid tensor triangulated subcategory generated by the motives of $E$. As a corollary of Theorem 8.3 in \cite{C}, if 
$\ml{E}_{ell}^*$ is cohomological connected, i.e., Conjecture \ref{main c} and Conjecture \ref{BS vanishing conjecture} hold for $E$, then there exists an abelian category of mixed motives for $E$. In the last section, we show that the embedding of the triangulated category of mixed Tate motives into motives for $E$ can be understood as a map between the derived categories of dg modules, which is induced by the inclusion of a sub-algebra $\widehat{\ml{N}}$ of $\ml{E}_{ell}^*$ to $\ml{E}_{ell}^*$ itself.
\

\section*{Acknowledgements}
This paper is part of the author’s Ph.D. dissertation written at Universit\"{a}t Duisburg-Essen. I am grateful to my advisor Professor Marc Levine for his constant guidance, encouragement and patience during this work. I would like to thank Giuseppe Ancona, Spencer Bloch, Owen Patashnick, Markus Spitzweck, Rin Sugiyama and Tomohide Terasoma for many helpful discussions. Finally, I'd like to thank the referee for useful comments.

\noindent
\textbf{Notations and Conventions:}
\

\noindent Let $k$ be a base field with characteristic zero. 
\

\noindent $\mathbf{Sch}_{k}$: the category of separated schemes (of finite type) over $k$.
\

\noindent $\Sm$: the category of smooth varieties over $k$.
\

\noindent $\Shk$: the category of Nisnevich sheaves with transfers over $k$.
\

\noindent $\Ab$: the category of Abelian groups.
\

\noindent For any additive category $M$, we let $C(M)$ denote the category of unbounded chain complexes over $M$. 
\

\noindent Next we use some notations defined in \cite{C, L}.
\

\begin{enumerate}
\item
Given an Adams cdga $\ml{N}$, we denote the category of cell modules (resp. finite cell modules) over $N$ defined in section 1.4 of \cite{L} by $\ml{CM}_{\ml{N}}$ (resp. $\ml{CM}_{\ml{N}}^f$). Denote the derived category of Adams graded dg $\ml{N}$-modules by $\ml{D}_{\ml{N}}$, which is defined in section 1.4 of \cite{L}. Denote the full subcategory with objects isomorphic in $\ml{D}_{\ml{N}}$ to a finite cell module by $\ml{D}_{\ml{N}}^f$.
\item
Given a cdga $E$ over $GL_2$ (resp. $\mb{T}_{\mb{K}}$), defined in Definition 2.4 in \cite{C}, we we denote the category of cell modules over $\ml{E}$ defined in Definition 2.9 of \cite{C} by $\ml{CM}^{GL_2}_{\ml{E}}$ (resp. $\ml{CM}^{\mb{T}_{\mb{K}}}_{\ml{E}}$). Denote the derived category of dg $\ml{E}$-modules by $\ml{D}^{GL_2}_{\ml{E}}$ (resp. $\ml{D}^{\mb{T}_{\mb{K}}}_{\ml{E}}$), which is defined in Definition 3.2 of \cite{C}. 
\end{enumerate}

%%%%%%%%%%%%%%%%%%%%%%%%%

\section{Motives for an elliptic curve} \label{motives for an ell}
In $\DM_{gm}(k, \mb{Q})$, the motive of $E$ will decompose into:
$$M(E) = \mb{Q} \oplus M_{1}(E)[1] \oplus \mb{Q}(1)[2].$$
Recall that $\DM^{eff}_{gm}(k, \mb{Q})$ is a $\mb{Q}$-linear tensor category. Using the results in Section 1.4 of \cite{Del02} we have the following decomposition of $M_1(E)^{\otimes n}$ (also in the category of Chow motives):
\begin{equation} \label{dec of elliptic motive I}
M_{1}(E)^{\otimes n} \cong \bigoplus_{|\lambda|=n} V_{\lambda} \otimes S_{\lambda}(M_{1}(E)),
\end{equation}
where $S_{\lambda}$ is the Schur functors\footnote{For the definition of Schur functor and notations of partitions, we refer to Section 1.3 and 1.4 of \cite{Del02}.} associated to $\lambda$, a partition of $n$. The index set runs through all partitions of $n$ and $V_{\lambda}$ is the multiplicity space.

%%%%%%% 
\begin{lem} \label{Basic decom of elliptic motive}
Let $E$ be an elliptic curve over $k$. Then we have $S_{\lambda}(M_{1}(E)) = 0$ if $\lambda = (n_1, n_2, \cdots, n_r)$ with $r \geq 3$ and $\wedge^{2}M_1(E) = \mb{Q}(1)$. 
In other words, equality (\ref{dec of elliptic motive I}) can be written as:
\begin{equation} 
M_{1}(E)^{\otimes n} \cong \bigoplus_{\lambda = (a+b,b), a+2b = n} V_{\lambda} \otimes Sym^{a}(M_{1}(E))(b).
\end{equation}
\end{lem}
\begin{proof}
By Proposition 20.1 in \cite{MVW}, we know that the category of effective Chow motives embeds contravariantly into $\DM^{eff}_{gm}(k, \mb{Q})$. Let us denote this functor by $\Phi$. In the category of Chow motives, we have the following decomposition of  the Chow motive of $E$:
$$h(E) = h^{0}(E) \oplus h^{1}(E) \oplus h^2(E).$$
Note that the image of $h^1(E)$ under $\Phi$ is $M_1(E)[1]$. Using Theorem 4.2 in \cite{Kim}, we get:
$$Sym^{i}h^{1}(E) = 0 \ \ \text{if} \ \ \ i \geq 3.$$
and
$$Sym^{2}h^1(E) = \mb{L}.$$
Here $\mb{L}$ is the Lefschetz motive in the category of Chow motives. Recall that the image of $\mb{L}$ under $\Phi$ is $\mb{Q}(1)[2]$(Remark 20.2 in \cite{MVW}). Because $\Phi$ is a tensor functor, using commutative constraint in $\DM^{eff}_{gm}(k, \mb{Q})$, we have:
\begin{equation} \nonumber
\Phi(Sym^i h^1(E)) = Sym^i (M_{1}(E)[1]) = (\wedge^iM_1(E))[i] 
\end{equation}
This implies that:
$$\wedge^i M_1(E) = 0 \ \ \ \text{if} \ \ \ i \geq 3,$$
and
$$\wedge^2 M_1(E) = \mb{Q}(1).$$
Given $\lambda = (n_1, n_2, \cdots, n_r)$ , by the definition of Young symmetrizer, we know that: $S_{\lambda}(M_1(E))$ is a direct summand of $\wedge^{m_1}M_1(E) \otimes \cdots \otimes \wedge^{m_s}M_1(E)$, where $(m_1, \cdots, m_s) = \lambda^t$\footnote{Here $\lambda^t$ is the transpose (or conjugate) of $\lambda$, which is defined by interchanging rows and columns in the Young diagram associated to $\lambda$.}. When $r \geq 3$, then we have $m_1 \geq 3$. By the above computation, we obtain that $S_{\lambda}(M_1(E)) = 0$.
\end{proof}
\begin{df} \label{DMEM_E}
Given an elliptic curve $E$, the full idempotent complete rigid tensor triangulated subcategory of $\DM_{gm}(k, \mb{Q})$ generated by $M(E)$ is denoted by $\DMEM(k, \mb{Q})_E$. 
\end{df}
\begin{rmk}
We remark that $\DMEM(k, \mb{Q})_E$ contains the category of mixed Tate motives because of the decomposition of the motive of $E$ in the beginning of this section.
\end{rmk}
\begin{rmk} \label{Basic decom of elliptic motiveII}
If $E$ is an elliptic curve with CM, we let $\mb{K} = End(E) \otimes \mb{Q}$, which is an imaginary quadratic field. Then we will consider $\DMEM(k, \mb{K})_E$. We recall that $M_1(E)_{\mb{K}}$ is decomposed as a direct sum of two motives $M$ and $\bar{M}$ in $\DM_{gm}(k, \mb{K})$. See Proposition 7.2 in \cite{A}. This decomposition is induced by the action of $\mb{K}$. For a given two-dimensional rational vector space $\V$, viewed as a $\mathbf{Res}_{\mb{K}/\mb{Q}}\mb{G}_m$-representation, then we have a decomposition as before:
\[
(\V_\mb{K})^{\otimes n} \cong \bigoplus_{a+2b = n, a, b \in \mb{Z}_{\geq 0}} V_\lambda \otimes Sym^a(\V_\mb{K})(b).
\]
Furthermore, the piece 
\[
c_n (\V_\mb{K})^{\otimes n} = V_{(n,0)} \otimes Sym^n(\V_\mb{K}) \cong \bigoplus_{i+j = n, i, j \in \mb{Z}} (V_{(n,0)} \otimes V^{\otimes i} \otimes \bar{V}^{\otimes j}) \otimes V_{i, j},\footnote{We view $c_n$ as an idempotent in $End((\V_\mb{K})^{\otimes n})$, which lies in $End_{\mathbf{Res}_{\mb{K}/\mb{Q}}\mb{G}_m \otimes \mb{K}}((\V_{\mb{K}})^{\otimes n})$.}
\]
where $V_{(n,0)} \otimes V^{\otimes i} \otimes \bar{V}^{\otimes j}$ are pairwise non-isomorphic irreducible representations over a $\mb{K}$-algebra $End_{\mathbf{Res}_{\mb{K}/\mb{Q}}\mb{G}_m \otimes \mb{K}}((\V_{\mb{K}})^{\otimes n})$ and $V_{i,j}$ are pairwise non-isomorphic irreducible representation over $\mathbf{Res}_{\mb{K}/\mb{Q}}\mb{G}_m \otimes \mb{K} = \mb{T}_{\mb{K}}$. For simplicity, we delete $V_{i,j}$ and one may think that both $V$ and $\bar{V}$ are endowed with the $\mb{G}_m$-action. In fact, $End_{\mathbf{Res}_{\mb{K}/\mb{Q}}\mb{G}_m \otimes \mb{K}}((\V_{\mb{K}})^{\otimes n})$ is a special case defined in the Section 3.9 of \cite{A}, which is called $B_{n, \mb{K}}$. Ancona's main result -- Theorem 4.1 in \cite{A} implies that the decomposition like Lemma \ref{Basic decom of elliptic motive} is holding for the CM elliptic motives:
\[
c_n(M_1(E)_{\mb{K}})^{\otimes n} \cong \bigoplus_{i+j = n, i, j \in \mb{Z}} V_{(n,0)} \otimes M^{\otimes i} \otimes \bar{M}^{\otimes j}.
\]
\end{rmk}
\begin{df}  \label{r-th van}
\begin{itemize}
\item[1)]
We say the $0$-th vanishing property holds for $E$ if :
\begin{equation} \nonumber
(\mathbf{Non-CM \ case}) \ \ \ \ \ \  Hom_{\DM_{gm}(k, \mb{Q})}(Sym^{2i}M_1(E), \mb{Q}(i)[j]) \cong 0,
\end{equation}
for any $j \in \mb{Z}_{\leq 0}$, any $i \in \mb{Z}_{> 0}$;
\begin{equation} \nonumber
(\mathbf{CM \ case}) \ \ \ \ \ Hom_{\DM_{gm}(k, \mb{K})}(M^{\otimes 2i}, \mb{K}(i)[j]) \cong Hom_{\DM_{gm}(k, \mb{K})}(\bar{M}^{\otimes 2i}, \mb{K}(i)[j]) \cong 0
\end{equation}
for any $j \in \mb{Z}_{\leq 0}$ and any $i \in \mb{Z}_{> 0}$
\item[2)]
Let $r$ be a positive integer. We say the $r$-th vanishing property holds for $E$ if 
\begin{equation} \nonumber
(\mathbf{Non-CM \ case}) \ \ \ \ \ \  Hom_{\DM_{gm}(k,\mb{Q})}(Sym^{2i+r}M_1(E), \mb{Q}(i)[j]) \cong 0,
\end{equation}
for any $j \in \mb{Z}$ such that $r+j \leq 0$ and any $i \in \mb{Z}_{\geq 0}$;
\begin{equation} \nonumber
(\mathbf{CM \ case}) \ \ \ \ \ Hom_{\DM_{gm}(k,\mb{K})}(M^{\otimes 2i+r}, \mb{Q}(i)[j]) \cong Hom_{\DM_{gm}(k,\mb{K})}(\bar{M}^{\otimes 2i+r}, \mb{Q}(i)[j]) \cong 0,
\end{equation}
for any $j \in \mb{Z}$ such that $r+j \leq 0$ and any $i \in \mb{Z}_{\geq 0}$.
\end{itemize}
 %such that $2i + r \geq 0$. 
\end{df}
\begin{conj}   \label{main c}
If $E$ be an elliptic curve over a field $k$ of characteristic zero, then $E$ has the $r$-th vanishing property for any non-negative integer $r$. 
\end{conj}
\begin{exe} \label{ex for van}
Assume that $E$ is an elliptic curve without CM, then we have:
\begin{equation}  \nonumber
Hom_{\DM_{gm}(k, \mb{Q})}(Sym^2 M_1(E), \mb{Q}(1)[*]) \cong 0.
\end{equation}
\begin{proof}
Notice that:
$$Hom_{\DM_{gm}(k, \mb{Q})}(Sym^2 M_1(E), \mb{Q}(1)[i]) \cong Hom_{\DM_{gm}(k, \mb{Q})}(Sym^2 M_1(E)[2], \mb{Q}(1)[i+2])$$ 
is a direct summand of $Hom_{\DM_{gm}(k, \mb{Q})}((M_1(E)[1])^{\otimes 2}, \mb{Q}(1)[i+2])$, therefore a direct summand of 
$$Hom_{\DM_{gm}(k, \mb{Q})}(M(E \times E), \mb{Q}(1)[i+2]).$$
It's well known that, (for example, chapter 3 in \cite{MVW}):
\begin{itemize}
\item
When $i = 0$, $Hom_{\DM_{gm}(k, \mb{Q})}(M(E \times E), \mb{Q}(1)[2]) \cong Pic(E \times E)$;
\item
When $i = -1$, $Hom_{\DM_{gm}(k, \mb{Q})}(M(E \times E), \mb{Q}(1)[1]) \cong k^{*}$;
\item
Otherwise,  $Hom_{\DM_{gm}(k, \mb{Q})}(M(E \times E), \mb{Q}(1)[2+i]) \cong 0$.
\end{itemize}
Notice that $Hom_{\DM_{gm}(k, \mb{Q})}(\mb{Q}, \mb{Q}(1)[1]) \cong k^{*}$ is a direct summand of 
\

\noindent $Hom_{\DM_{gm}(k, \mb{Q})}(M(E \times E), \mb{Q}(1)[1])$, which implies that:
$$Hom_{\DM_{gm}(k, \mb{Q})}((M_1(E)[1])^{\otimes 2}, \mb{Q}(1)[i]) \cong 0 \ \ \ \text{if} \ \ \ i \neq 0.$$
Then
\begin{equation}  \label{1}
\begin{split}  
& Hom_{\DM_{gm}(k, \mb{Q})}(M_{1}(E^2), \mb{Q}(1)) \\
\cong & Hom_{\DM_{gm}(k, \mb{Q})}(Sym^2 M_1(E), \mb{Q}(1)) \oplus Hom_{\DM_{gm}(k, \mb{Q})}(\mb{Q}(1), \mb{Q}(1)) \\
\cong & Hom_{\DM_{gm}(k, \mb{Q})}(Sym^2 M_1(E), \mb{Q}(1)) \oplus  \mb{Q}
\end{split}
\end{equation}
On the other hand, we have:
\begin{equation} \label{2}
\begin{split}
& Hom_{\DM_{gm}(k, \mb{Q})}(M_{1}(E^2), \mb{Q}(1)) 
\cong  Hom_{\DM_{gm}(k, \mb{Q})}(M_1(E), M_1(E))  
\cong  Hom_{Ab_{\mb{Q}}}(E, E),
\end{split}
\end{equation}
where $Ab_{\mb{Q}}$ is in the category of abelian varieties up to isogeny.
For the first isomorphism, we use the facts that the dual motive of $M_{1}(E)$ is $M_1(E)(-1)$ and the properties of internal hom in $\DM_{gm}(k, \mb{Q})$. For the second one, we use the fact that the category of abelian varieties up to isogeny fully embeds into $\DM^{eff}(k, \mb{Q})$, for example, Proposition 2.2.1 in \cite{AEH}.
\

Putting together (\ref{1}) and (\ref{2}), we get:
$$Hom_{\DM_{gm}(k, \mb{Q})}(Sym^2 M_1(E), \mb{Q}(1)) \oplus  \mb{Q} \cong Hom_{Ab_{\mb{Q}}}(E, E).$$
If $E$ is an elliptic curve without CM, then we have: $Hom_{Ab_{\mb{Q}}}(E, E) \cong \mb{Q}$, which implies that:
$$Hom_{\DM_{gm}(k, \mb{Q})}(Sym^2 M_1(E), \mb{Q}(1))  \cong 0.$$
\end{proof}
\end{exe}
\begin{exe}  \label{ex for nonvan}
We let $E$ be an elliptic curve over $k$ with CM. After extending the rational coefficients for motives to $\mb{K}$-coefficients,  the similar analysis as above tells us
\begin{equation}  \nonumber
Hom_{\DM_{gm}(k, \mb{K})}(Sym^2 M_1(E)_{\mb{K}}, \mb{K}(1)[*]) \cong \mb{K}.
\end{equation}
Then the identification
$Sym^2 M_1(E)_{\mb{K}} \cong Sym^2(M \oplus \bar{M}) \cong M \otimes M \oplus M \otimes \bar{M} \oplus \bar{M} \otimes \bar{M}$
implies that:
\begin{itemize}
\item[i).]
$Hom_{\DM_{gm}(k, \mb{K})}(M \otimes M, \mb{K}(1)[*]) \cong 0$,
\item[ii).]
$Hom_{\DM_{gm}(k, \mb{K})}(\bar{M} \otimes \bar{M}, \mb{K}(1)[*]) \cong 0$,
\item[iii)]
$Hom_{\DM_{gm}(k, \mb{K})}(M \otimes \bar{M}, \mb{K}(1)[*]) \cong \mb{K}$.
\end{itemize}
In fact, we have $M \otimes \bar{M} \cong \mb{K}(1)$ in $\DM_{gm}(k, \mb{K})$.
\end{exe}
%%%%%%%%%%%% 本节全部用的\DM_{gm}

%\noindent In this chapter, fix a base field $k$ of character zero.
\section{Friedlander-Suslin complexes and their alternating versions} \label{FS}
\begin{df}
Take $Y$ in $\mathbf{Sm}_{k}$ and $X$ in $\mathbf{Sch}_{k}$.
Let $z_{q.fin}(X)(Y)$ be the free abelian group generated by integral closed subschemes $W \subset Y \times_{k} X$ such that $p_{1}: W \longrightarrow Y$ is quasi-finite and dominant over an irreducible component of $Y$.
\end{df}
\begin{rmk}
We recall that forany $i \in \mb{Z}$, the Friedlander-Suslin complexes $\mathbb{Z}^{SF}(i)$ is defined by: $$\mathbb{Z}^{SF}(i) = C_{*}z_{q.fin}(\mathbb{A}^{i})[-2i].$$
\end{rmk}

In order to define the alternating versions of Friedlander-Suslin complexes,  we define $C^{cb}_{*}(\mathcal{F})$ and $C^{Alt}_{*}(\mathcal{F})$ for every $\mathcal{F}$ a presheaf over $\mathbf{Sm}_{k}$.

\begin{df}
Let $X \in \mathbf{Sm}_{k}$ and $\mathcal{F}$ as above. Let $C^{cb}_{n}(\mathcal{F})$ be the presheaf
$$C^{cb}_{n}(\mathcal{F})(X) = \mathcal{F}(X \times \Box^{n})/ \sum^{n}_{j=1} \pi^{*}_{j}(\mathcal{F}(X \times \Box^{n-1})).$$
and the differential is given by:
$$d_{n} = \sum^{n}_{j=1}(-1)^{j-1}\mathcal{F}(\iota_{j,1}) - \sum^{n}_{j=1}(-1)^{j-1}\mathcal{F}(\iota_{j,0}).$$
\end{df}

If $\mathcal{F}$ is a Nisnevich presheaf (sheaf, with transfers), then $C^{cb}_{*}(\mathcal{F})$ is a complex of Nisnevich prsheaves (sheaves, with transfers). One can extend the construction to complexes of sheaves (with transfers) by taking the total complex. We can define $C^{Alt}_{*}(\mathcal{F})$ as a subcomplex of $C^{cb}_{*}(\mathcal{F}) \otimes \mathbb{Q}$ by taking the alternating elements in $C^{cb}_{*}(\mathcal{F})(Y)$ %as in Definition \ref{def of alt} 
for every $Y \in \mathbf{Sm}_{k}$.

\begin{rmk}
There is another definition of the alternating complex without taking the quotient by the degenerate cycles. See Remark 4.1.2 in \cite{L}.
\end{rmk}

The following theorem is concerning some comparison results about the above constructions. The proof can be found in Section 2.5 in \cite{L3}.
\begin{thm} \label{Comp alt}
Let $\mathcal{F}$ be a complex of presheaves on $\mathbf{Sm}_{k}$.
\

\begin{itemize}
\item
There is a natural isomorphism $C_{*}(\mathcal{F}) \cong C^{cb}_{*}(\mathcal{F})$ in the derived category of pre-sheaves on $\mathbf{Sm}_{k}$. If $\mathcal{F}$ is a complex of presheaves with transfer, there is also an isomorphism $C_{*}(\mathcal{F}) \cong C^{cb}_{*}(\mathcal{F})$ in the derived category of presheaves with transfers $D(\PST)$.
\item
The inclusion $C^{Alt}_{*}(\mathcal{F})(Y) \subset C^{cb}_{*}(\mathcal{F})_{\mathbb{Q}} (Y)$ is a quasi-isomorphism for all $Y \in \mathbf{Sm}_{k}$.
\end{itemize}
\end{thm}
As a corollary, we take $\ml{F}$ to be $z_{q.fin}(\mb{A}^i)$ and get the alternating versions of Friedlander-Suslin complexes, which are quasi-isomorphic to the original ones.

\section{The cycle algebra for an elliptic curve} \label{elliptic cycle algebra}
\noindent Let $E$ be an elliptic curve defined over a base field $k$ of characteristic zero. Given a positive integer $a$, we denote the $a$-th power of $E$ by $E^{a}$. 
\

\begin{df}
The sign character $\mathbf{sgn}: \mb{Z}/2\mb{Z} \to \{\pm 1\}$  extends to the map 
$$\rho: (\mb{Z}/2\mb{Z})^a \to \{\pm 1\}^a$$
by
$$\rho(\eta_1, \cdots, \eta_a) = \{\mathbf{sgn}(\eta_1), \cdots,  \mathbf{sgn}(\eta_a)\}$$
for $(\eta_1, \cdots, \eta_a) \in (\mb{Z}/2\mb{Z})^a$.
\end{df}
The group $\Gamma_n = (\mb{Z}/2\mb{Z})^{a} \rtimes \Sigma_{a}$ acts on $E^{a}$ in the following way: $\Sigma_{a}$ permutes the components of $E^{a}$ and the $i$-th generator $(0, \cdots, 1, \cdots, 0)$ in $(\mb{Z}/2\mb{Z})^{a}$ acts on the $i$-th component $E$ of $E^{a}$ by the inversion, i.e., $x \xrightarrow{\sigma_{i}} -x$. In the following, for a given $g \in \Gamma_n$, we denote the action of $g$ on an algebraic cycle $Z$ by $g(Z)$. For $i \in \mb{Z}$,  we define a subgroup of $C^{Alt}_{i}(z_{q.fin}(\mb{A}^{b}))(E^{a})$:
\begin{equation} \nonumber
\begin{split}
& C^{Alt , -}_{i}(z_{q.fin}(\mb{A}^{b}))(E^{a}) 
=  \{Z \in C^{Alt}_{i}(z_{q.fin}(\mb{A}^{b}))(E^{a}) | g(Z) = \rho(g)Z \  \forall g \in (\mb{Z}/2\mb{Z})^{a} \}.
\end{split}
\end{equation}
We denote the corresponding cycle complex by $C^{Alt , -}_{*}(z_{q.fin}(\mb{A}^{b}))(E^{a})$. Given $\sigma \in \mb{Q}[\Sigma_{a}]$, define $$Z \bullet \sigma = sgn(\sigma)\sigma^{-1}(Z)$$ for $Z \in C^{Alt , -}_{i}(z_{q.fin}(\mb{A}^{b}))(E^{a})$. This makes $C^{Alt , -}_{i}(z_{q.fin}(\mb{A}^{b}))(E^{a})$ into a right $\mb{Q}[\Sigma_{a}]$-module. We also have the action of the symmetric group $\Sigma_{b}$ on $C^{Alt , -}_{i}(z_{q.fin}(\mb{A}^{b}))(E^{a})$, by permuting the coordinates of $\mb{A}^{b}$. Taking the symmetric sections with respect to the action of $\Sigma_{b}$, we get a sub-complex $\widetilde{C}^{Alt , -}_{i}(z_{q.fin}(\mb{A}^{b}))(E^{a})$ of $C^{Alt , -}_{i}(z_{q.fin}(\mb{A}^{b}))(E^{a})$.
\

Assume that $E$ is an elliptic curve with complex multiplication and recall that $\mb{K} = End_k(E) \otimes \mb{Q}$. In this case, we consider the above cycle complexes with $\mb{K}$-coefficients rather than $\mb{Q}$-coefficients. By a slight of abuse of notations, we still use the same symbol for cycle complexes and representations. 
\

Notation: For $i < 0$, $V^{\otimes i} = (V^{\vee})^{\otimes - i}$, where $V^{\vee}$ is the dual representation of $V$. We also use this notation for motives.
\

\begin{df} \label{E^{a,b}} %%%%%% 
Let $a, b$ be integers such that $a \geq b, a \geq 0$. 
\begin{itemize}
\item
For $i \in \mb{Z}$ and $E$ an elliptic curve without CM, we define:
$$\mathcal{E}_{a, b}^{i} = \widetilde{C}^{Alt , -}_{a - 2b - i}(z_{q.fin}(\mb{A}^{a - b}))(E^{a}) \otimes_{\mb{Q}[\Sigma_{a}]}  \V^{\otimes a}(b-a).$$
Here $\V$ is the fundamental representation of $GL_2$ and $\V^{\otimes a}(b-a) = \V^{\otimes a} \otimes det^{\otimes b-a}$.
\item
For $i \in \mb{Z}$ and $E$ an elliptic curve with CM, we define:
$$\mathcal{E}_{a, b}^{i} = \widetilde{C}^{Alt , -}_{a - 2b - i}(z_{q.fin}(\mb{A}^{a - b}))(E^{a}) \otimes_{B_{n, \mb{K}}} c_n(\V^{\otimes a})(b-a).$$
Here $\V$ is the fundamental representation of $GL_2 \otimes \mb{K}$, $\V^{\otimes a}(b-a) = \V^{\otimes a} \otimes det^{\otimes b-a}$. We recall that $B_{n, \mb{K}}$ and $c_n$ are defined in Remark \ref{Basic decom of elliptic motiveII}.
\end{itemize}
\end{df}

\begin{rmk} \label{ext product}
We first collect some facts.
\begin{enumerate}
\item
Using the external product of cycles, we define a map:
$$\widetilde{C}^{Alt , -}_{a - 2b - i_{1}}(z_{q.fin}(\mb{A}^{a-b}))(E^{a}) \otimes_{\mb{Q}} \widetilde{C}^{Alt , -}_{c - 2d - i_{2}}(z_{q.fin}(\mb{A}^{c-d}))(E^{c})$$
$$\xrightarrow{\boxtimes} \widetilde{C}^{Alt , -}_{a - 2b + c - 2d - i_{1} - i_{2}}(z_{q.fin}(\mb{A}^{a-b+c-d}))(E^{a+c}),$$
which sends $Z_{1} \otimes Z_{2}$ to $(-1)^{c(a - 2b - i_{1})}m(Z_{1} \times Z_{2})$.
%%%%% This is because box and E both has cohomological degree 1.
Here $m$ is the map
$$E^{a} \times \mb{A}^{a-b} \times \square^{a-2b-i_{1}} \times E^{c} \times \mb{A}^{c-d} \times \square^{c-2d-i_{2}}
\to E^{a+c} \times \mb{A}^{a-b+c-d} \times \square^{a-2b+c-2d-i_{1}-i_{2}},$$
changing the positions of the factors.
\item
We have the map of $GL_{2}$ representations: $\V^{a} \otimes \V^{c} \to \V^{a+c}$.
\end{enumerate}
\end{rmk}
In the following, we want to define a product map $\ml{E}^{*}_{a, b} \otimes \ml{E}^{*}_{c,d} \to \ml{E}^{*}_{a+c, b+d}$. For simplicity, we denote $\widetilde{C}^{Alt , -}_{a - 2b - i}(z_{q.fin}(\mb{A}^{a-b}))(E^{a})$ by $C_{a, b}^{i}$. Let us only explain the non-CM case. We have:
$$(C_{a,b}^{i} \otimes_{\mb{Q}[\Sigma_{a}]} \V^{a}(b-a)) \otimes_{\mb{Q}} (C_{c,d}^{j} \otimes_{\mb{Q}[\Sigma_{c}]} \V^{c}(d-c))
= (C_{a,b}^{i} \otimes_{\mb{Q}} C_{c,d}^{j}) \otimes_{\mb{Q}[\Sigma_{a} \times \Sigma_{c}]} (\V^{a}(b-a) \otimes_{\mb{Q}} \V^{c}(d-c)).$$
Using the external product of cycles and $GL_{2}$-representations (see Remark \ref{ext product}), we have a map:
$$(C_{a,b}^{i} \otimes_{\mb{Q}} C_{c,d}^{j}) \otimes_{\mb{Q}[\Sigma_{a} \times \Sigma_{c}]} (\V^{a}(b-a) \otimes_{\mb{Q}} \V^{c}(d-c))
\to C_{a+c,b+d}^{i+j} \otimes_{\mb{Q}[\Sigma_{a} \times \Sigma_{c}]} \V^{a+c}(b-a+d-c).$$
The injection of groups $\Sigma_{a} \times \Sigma_{c} \to \Sigma_{a+c}$ induces a map $\mb{Q}[\Sigma_{a} \times \Sigma_{c}] \to \mb{Q}[\Sigma_{a+c}]$.
Note that both $C_{a+c,b+d}^{i+j}$ and $\V^{a+c}(b-a+d-c)$ are $\mb{Q}[\Sigma_{a+c}]$ modules, and their $\mb{Q}[\Sigma_{a+c}]$ modules structures are compatible with their $\mb{Q}[\Sigma_{a} \times \Sigma_{c}]$ module structure coming from the respective external products. This tells us that there is a map: $$C_{a+c,b+d}^{i+j} \otimes_{\mb{Q}[\Sigma_{a} \times \Sigma_{c}]} \V^{a+c}(b-a+d-c) \to C_{a+c,b+d}^{i+j} \otimes_{\mb{Q}[\Sigma_{a+c}]} \V^{a+c}(b-a+d-c).$$
Putting these maps together, we get a map:
\begin{equation} \label{mul-str}
\begin{split}
&(C_{a,b}^{i} \otimes_{\mb{Q}[\Sigma_{a}]} \V^{a}(b-a)) \otimes_{\mb{Q}} (C_{c,d}^{j} \otimes_{\mb{Q}[\Sigma_{c}]} \V^{c}(d-c)) 
 \xrightarrow{\mu^{a,b}_{c,d}} C_{a+c,b+d}^{i+j} \otimes_{\mb{Q}[\Sigma_{a+c}]} \V^{a+c}(b-a+d-c).
\end{split}
\end{equation}
\begin{rmk}
Similarly, the above construction of multiplicative maps can be also applied to the case of an elliptic curve with CM.
\end{rmk}
\begin{rmk}
We will use the following identification in the next lemma.
\

Let $G$ be a finite group and $V$ (resp. $W$)  be a right (resp. left) $\mb{Q}[G]$ module, then:
$$V \otimes_{\mb{Q}[G]} W = (V \otimes_{\mb{Q}} W)_{G},$$
The right hand means the following: The right $\mb{Q}[G]$ module $V$ can be considered as a left module, $g \bullet v \doteq v \bullet g^{-1}$. $V \otimes_{\mb{Q}} W$ is considered as a left module. Then take the $G$ co-invariant part. This is even true for any algebra.
\end{rmk}
\begin{lem} \label{mul}
The product structure defined in (\ref{mul-str}) is associative and graded commutative. More precisely, $(-1)^{ij} \mu^{c,d}_{a,b} \circ \tau = \mu^{a,b}_{c,d}$, where $\tau$ is the map
$\ml{E}_{a,b}^{i} \otimes \ml{E}_{c,d}^{j} \xrightarrow{\tau} \ml{E}_{c,d}^{j} \otimes \ml{E}_{a,b}^{i}$ changing two factors.
\end{lem}
\begin{proof}
For the associativity part, it's the direct result of the associativity of external products of cycles and representations and compatibility of actions of symmetric groups.
\

For the commutativity part, we need to check that $(-1)^{ij} \mu^{c,d}_{a,b} \circ \tau = \mu^{a,b}_{c,d}$. Take $Z_{1} \otimes W_{1}$, where $Z_{1} \in C_{a,b}^{i}$, $W_{1} \in V^{a}(b-a)$. Similarly take $Z_{2} \otimes W_{2}$, where $Z_{2} \in C_{c,d}^{j}$, $W_{2} \in V^{c}(d-c)$. Let $\sigma$ be the element in $\mb{Q}[\Sigma_{a+c}]$ which permutes the first $a$ elements with the last $c$ elements. Let $\delta$ act on $\mb{A}^{a-b} \times \square^{a - 2b - i} \times \mb{A}^{c-d} \times \square^{c - 2d - j}$ by permuting $\mb{A}^{a-b}$ and $\mb{A}^{c-d}$, and permuting $\square^{a-2b-i}$ and $\square^{c-2d-j}$.
\

Also use $\boxtimes$ to denote the external product of modules. 
\

Then, in $C_{a+c, b+d}^{i+j} \otimes_{\mb{Q}} \V^{a+c}(b-a+d-c)$, we have:
\begin{equation} \nonumber
\begin{split}
&\delta \sigma ((Z_{1} \boxtimes Z_{2}) \otimes_{\mb{Q}} (W_{1} \boxtimes W_{2}))  \\
=\ & \delta (\sigma (Z_{1} \boxtimes Z_{2})) \otimes_{\mb{Q}} (\sigma(W_{1} \boxtimes W_{2})) \\
= \ &(-1)^{c(a-i)+a(c-j)+(a - i)(c-j)}(Z_{2} \boxtimes Z_{1}) \otimes_{\mb{Q}} (W_{2} \boxtimes W_{1}) \\
= \ & (-1)^{c(a-i)+a(c-j)+(a - i)(c-j)}(Z_{2} \boxtimes Z_{1}) \otimes_{\mb{Q}} (W_{2} \boxtimes W_{1}) \\
=\  & (-1)^{ac+ij}(Z_{2} \boxtimes Z_{1}) \otimes_{\mb{Q}} (W_{2} \boxtimes W_{1}).
\end{split}
\end{equation}
Here we use $\delta(Z) = sgn(\delta) Z$.
This implies that the image of
$$\mu^{a,b}_{c,d}((Z_{1} \otimes W_{1}) \otimes_{\mb{Q}} (Z_{2} \otimes W_{2})) - (-1)^{ij}(\mu^{c,d}_{a,b} \circ \tau)((Z_{1} \otimes W_{1}) \otimes_{\mb{Q}} (Z_{2} \otimes W_{2}))$$
in $\ml{E}_{a+c,b+d}^{i+j}$ is the same as
\begin{equation} \nonumber
\mu^{a,b}_{c,d}((Z_{1} \otimes W_{1}) \otimes_{\mb{Q}} (Z_{2} \otimes W_{2})) - (-1)^{ac} \sigma (\mu^{c,d}_{a,b}((Z_{1} \otimes W_{1}) \otimes_{\mb{Q}} (Z_{2} \otimes W_{2}))),
\end{equation}
i.e.,
\begin{equation} \nonumber
\mu^{a,b}_{c,d}((Z_{1} \otimes W_{1}) \otimes_{\mb{Q}} (Z_{2} \otimes W_{2})) -  \sigma \bullet (\mu^{c,d}_{a,b}((Z_{1} \otimes W_{1}) \otimes_{\mb{Q}} (Z_{2} \otimes W_{2}))),
\end{equation}
which is zero in $\ml{E}_{a+c,b+d}^{i+j}$. This implies the graded commutativity.
\end{proof}
For simplicity, we denote the multiplication $\mu^{a,b}_{c,d}$ by $\bullet$.
\begin{lem} \label{diff}
Given $Z_{1} \otimes W_{1} \in C_{a,b}^{i} \otimes_{\mb{Q}[\Sigma_{a}]} \V^{a}(b-a)$, $Z_{2} \otimes W_{2} \in C_{c,d}^{j} \otimes_{\mb{Q}[\Sigma_{b}]} \V^{c}(d-c)$, then we have:
$$d^{i+j}_{a+c,b+d}((Z_{1} \otimes W_{1}) \bullet (Z_{2} \otimes W_{2}))$$
$$= (d^{i}_{a,b}(Z_{1} \otimes W_{1})) \bullet (Z_{2} \otimes W_{2}) +  (-1)^{i}(Z_{1} \otimes W_{1}) \bullet(d^{j}_{c,d}(Z_{2} \otimes W_{2})),$$
where $d^{i}_{a,b}$ is the map $C_{a, b}^{i} \otimes_{\mb{Q}[\Sigma_{a}]} \V^{a}(b-a) \xrightarrow{d \otimes id} C_{a, b}^{i+1} \otimes_{\mb{Q}[\Sigma_{a}]} \V^{a}(b-a)$.
\end{lem}
\begin{proof}
One may check by definition.
\end{proof}
Assume that $E$ is an elliptic curve without CM. By Lemma \ref{mul} and Lemma \ref{diff}, our products
$$\ml{E}^{*}_{a, b} \otimes \ml{E}^{*}_{c,d} \to \ml{E}^{*}_{a+c, b+d}$$
give $\oplus_{a \geq b \geq 0}\ml{E}^{*}_{a,b}$ the structure of a bi-graded commutative differential graded algebra in $GL_2$-representations.
\begin{rmk}
If $E$ is an elliptic curve with CM, then the representations in each $\ml{E}^*_{a,b}$ are viewed as representations over $\mb{G}_{m, \mb{K}} \times \mb{G}_{m, \mb{K}} = T_{\mb{K}}$. The determinant representation of $T_{\mb{K}}$ means the pullback of the determinant representation of $GL_2$ along the embedding of $T_{\mb{K}} \to GL_2$. Under the action of the Galois group $\mathbf{Gal}(\mb{K}/\mb{Q})$, we have the decomposition $\V = V \oplus \bar{V}$. In this case, the analogue of Lemma \ref{mul} and Lemma \ref{diff} hold, i.e., the products 
$$\ml{E}^{*}_{a, b} \otimes \ml{E}^{*}_{c,d} \to \ml{E}^{*}_{a+c, b+d}$$
give $\oplus_{a \geq b \geq 0}\ml{E}^{*}_{a,b}$ the structure of a bi-graded cdga in $T_{\mb{K}}$-representations.
\end{rmk}
\begin{exe} \label{E21}
Assume that $E$ is an elliptic curve without CM in this example. 
Let us use Example \ref{ex for van} to compute $\ml{E}^{*}_{2,1}$. According to our definition, we have:
$$\ml{E}^{*}_{2,1} = \widetilde{C}^{Alt , -}_{-*}(z_{q.fin}(\mb{A}^{1}))(E^{2}) \otimes_{\mb{Q}[\Sigma_{2}]} \V^{\otimes 2}(-1).$$
Notice that as a $GL_{2}$ representation, $V^{\otimes 2}(-1)$ decomposes as the direct sum of $Sym^{2}\V(-1)$ and $\mb{Q}$, both factors with multiplicity one.  Computing the corresponding cycle complexes, we get:
$$\ml{E}^{*}_{2,1} = (\widetilde{C}^{Alt , -}_{-*}(z_{q.fin}(\mb{A}^{1}))(E^{2}))^{sym} \otimes_{\mb{Q}} Sym^{2}\V(-1) 
\oplus (\widetilde{C}^{Alt , -}_{-*}(z_{q.fin}(\mb{A}^{1}))(E^{2}))^{alt} \otimes_{\mb{Q}} \mb{Q}.$$
%%% Hom(Sym^{2}M_{1}(E), \mb{Q}) = 0
Using Example \ref{ex for van}, we obtain that the first term of right hand side is quasi-isomorphic to zero. Similarly the second term is quasi-isomorphic to the trivial $GL_{2}$ representation, generated by a cycle of codimension one in $E^{2}$. If we denote the diagonal (resp. anti-diagonal) of $E \times E$ by $\Delta^{+}$ (resp. $\Delta^{-}$), then we can take this generator to be the cycle $\frac{1}{2}(\Delta^{+} - \Delta^{-})$.
\end{exe}
\begin{exe} \label{E21cm}
Assume that $E$ is an elliptic curve with CM now. Then by the above discussion, we have:
$$\ml{E}^{*}_{2,1} = (\widetilde{C}^{Alt , -}_{-*}(z_{q.fin}(\mb{A}^{1}))(E^{2}))^{sym} \otimes_{\mb{K}} (V \otimes \bar{V}(-1)).$$
Notice that this complex is quasi-isomorphic to the trivial representation $\mb{K}$, generated by the cycle $\frac{1}{2}(\Gamma_{\iota} - \Gamma_{\iota \circ (-1)})$, where $\Gamma_{\iota}$ denotes the graph of the complex multiplication $\iota$.
\end{exe}

(\textbf{Non-CM} case) Let $E$ be an elliptic curve without multiplication. We define a map:
\begin{equation} \label{colimit process}
\ml{E}^{*}_{a, b} \xrightarrow{\eta} \ml{E}^{*}_{a+2, b+1}
\end{equation}
by mapping $Z \otimes_{\mb{Q}} W \in C_{a,b}^{i} \otimes_{\Sigma_{a}} \V^{a}(b-a)$ to $(Z \times \frac{1}{2}(\Delta^{+} - \Delta^{-})) \otimes_{\mb{Q}} \phi(W)$, where $\phi$ is the composition of maps between $GL_{2}$ representations $\V^{a}(b-a) \to \V^{a}(b-a) \otimes_{\mb{Q}} \V^{2}(-1) \xrightarrow{\cong} \V^{a+2}(b-a-1)$. The first map is defined in the following way. Because $\V^{2}(-1) \cong Sym^{2}\V(-1) \oplus \mb{Q}$ as $GL_2$ representations, we have a natural injective map:
$$\V^{a}(b-a) = \V^{a}(b-a) \otimes \mb{Q} \to \V^{a}(b-a) \otimes \V^2(-1),$$
sending $1 \in \mb{Q}$ to $1 \in \mb{Q} \subset \V^2(-1)$.

\begin{rmk} \label{existence of beta}
Using the computation in Example \ref{E21} and Example \ref{E21cm}, the above definition is just the composition of maps:
$$\ml{E}_{a,b}^{*} \to \ml{E}_{a, b}^* \otimes_{\mb{Q}}((\widetilde{C}^{Alt , -}_{-*}(z_{q.fin}(\mb{A}^{1}))(E^{2}))^{alt} \otimes_{\mb{Q}} \mb{Q}) \to \ml{E}_{a,b}^{*} \otimes \ml{E}_{2,1}^{*} \xrightarrow{\bullet} \ml{E}^{*}_{a+2,b+1}.$$
\end{rmk}

(\textbf{CM} case) Assume that $E$ is with CM. Notice that $Sym^2 \V$ is not irreducible as an $T_{\mb{K}}$-representation. In fact, $Sym^2 \V$ is decomposed as a direct sum of $V^{\otimes 2}$, $\bar{V}^{\otimes 2}$ and $V \otimes \bar{V}$. As non-CM case, using Example \ref{E21cm} again together with
\begin{equation} \nonumber
\ml{E}^*_{a, b} \to \ml{E}^*_{a, b} \otimes ((\widetilde{C}^{Alt , -}_{-*}(z_{q.fin}(\mb{A}^{1}))(E^{2}))^{sym} \otimes_{\mb{K}} (V \otimes \bar{V}(-1))
\to \ml{E}_{a,b}^{*} \otimes \ml{E}_{2,1}^{*} \xrightarrow{\bullet} \ml{E}_{a+2,b+1}^{*},
\end{equation}
we get a map:
$$\tau: \ml{E}^*_{a, b} \to \ml{E}^*_{a+2, b+1}.$$

\begin{df}  \label{df of elliptic-alg}
Given $a \in \mb{Z}$, 
\begin{itemize}
\item
for an elliptic curve without CM, we define: $$\ml{E}_{a}^{*} =   \underrightarrow{lim}_{i \geq -a} \  \ml{E}^{*}_{-a+2i, -a+i}$$
where the colimits are taken over the map $\eta$.
\item
for an elliptic curve with CM, we define:$$\ml{E}_{a}^{*} =   \underrightarrow{lim}_{i \geq -a} \  \ml{E}^{*}_{-a+2i, -a+i},$$
where the colimits are taken over the map $\tau$.
\end{itemize}
\end{df}

\begin{rmk}
As $GL_{2}$ representations, every term of the complex $\ml{E}_{a}^{*}$ has pure Adams weight $a$. The reason for the process of taking colimit is to kill the infinite repeated information. Notice each irreducible $GL_{2}$ representation appear infinite times for the representation part of $\{ \ml{E}^*_{a,b} \}$, which take the same cycle complexes. For example, in $\ml{E}^{*}_{a, b}$ and $\ml{E}^{*}_{a+2, b+1}$, the $Sym^a \V(b-a)$-isotypical pieces appear in these both complexes. We will see these facts later in Corollary \ref{cycle algebra GL_2 identification}. The same thing also holds for the CM case.
\end{rmk}
\begin{df} \label{df of cycle algebra}
Define:
$$\ml{E}^{*} = \mb{Q} \oplus \bigoplus_{a \geq 1}\ml{E}_{a}^{*}.$$
and 
$$\ml{E}_{ell}^{*} = \bigoplus_{a \in \mb{Z}} \ml{E}^{*}_{a}.$$    %%% algebra stucture need the associativity of \ml{E}^{a,b}
\end{df}
\begin{rmk}
The products on $\ml{E}^{*}_{a, b}$ descend to products on $\ml{E}^{*}$ and $\ml{E}_{ell}^{*}$. We take the case of an elliptic curve without CM as an example. By the construction of the multiplication map, we have:
$$\ml{E}^{*}_{a + 2i,\ i} \otimes \ml{E}^{*}_{b+2j,\ j} \to \ml{E}^{*}_{a+b+2i+2j , \ i+j},$$
and the commutative diagram

\begin{xy}
(210,0)*+{\ml{E}^{*}_{a+2i, \ i} \otimes \ml{E}^{*}_{b+2j,\ j}}="v1";
(260,0)*+{\ml{E}^{*}_{a+b+2i+2j , \ i+j}}="v2";
(210,-20)*+{\ml{E}^{*}_{a+2i+2,\ i+1} \otimes \ml{E}^{*}_{b+2j,\ j}}="v3";
(260,-20)*+{\ml{E}^{*}_{a+b+2i+2j+2, \ i+j+1}.}="v4";
{\ar@{->} "v1";"v2"};{\ar@{->}^{\eta \otimes id} "v1";"v3"};{\ar@{->} "v3";"v4"};{\ar@{->}^{\eta} "v2";"v4"};
\end{xy}
\

Fix integers $b, j$. Using these diagrams for $i$ varying, we get a map:
$$\ml{E}^{*}_{a} \otimes \ml{E}^{*}_{b+2j, \ j} \to \ml{E}^{*}_{a+b}.$$
We also have the following commutative diagrams:
\begin{center}
\begin{xy}
(220,0)*+{\ml{E}^{*}_{a} \otimes \ml{E}^{*}_{b+2j,\ j}}="v1";(260,0)*+{\ml{E}^{*}_{b}}="v2";
(220,-20)*+{\ml{E}^{*}_{a} \otimes \ml{E}^{*}_{b+2j+2,\ j+1}}="v3";
{\ar@{->} "v1";"v2"};{\ar@{->}^{id \otimes \eta} "v1";"v3"};{\ar@{->} "v3";"v2"};
\end{xy}
\end{center}
and
\begin{center}
\begin{xy}
(210,0)*+{\ml{E}^{*}_{a+2i, \ i} \otimes \ml{E}^{*}_{2b+j,\ b}}="v1";(260,0)*+{\ml{E}^{*}_{a+b+2i+2j , \ i+j}}="v2";
(210,-20)*+{\ml{E}^{*}_{a+2i, \ i} \otimes \ml{E}^{*}_{2b+j+2,\ b+1}}="v3";(260,-20)*+{\ml{E}^{*}_{a+b+2i+2j+2, \ i+j+1}.}="v4";
{\ar@{->} "v1";"v2"};{\ar@{->}^{id \otimes \eta} "v1";"v3"};{\ar@{->} "v3";"v4"};{\ar@{->}^{\eta} "v2";"v4"};
\end{xy}
\end{center}
Then for $i, j \in \mb{Z}$, we get a map $\ml{E}^{*}_{i} \otimes \ml{E}^{*}_{j} \to \ml{E}^{*}_{i+j}$, which induces product structures on $\ml{E}$ and $\ml{E}_{ell}^{*}$. By Lemma \ref{mul} and Lemma \ref{diff}, these give $\ml{E}$ and $\ml{E}_{ell}^{*}$ the structures of commutative differential graded algebra objects in the category of $GL_2$ representations.
\end{rmk}

\section{Computations for the non-CM case} \label{sec-comp}
\begin{lem} \label{Comp of cycle} There are isomorphisms:
\begin{equation} \nonumber
\begin{split}
& H^{i}(\widetilde{C}^{Alt , -}_{a - 2b - *}(z_{q.fin}(\mb{A}^{a-b}))(E^{a})) \cong  Hom_{\DM_{gm}(k, \mb{Q})}((M_{1}(E))^{\otimes a}, \mb{Q}(a-b)[i]),
\end{split}
\end{equation}
for $i \in \mb{Z}$.
\end{lem}
\begin{proof}
For the proof, we need use bivariant cycle cohomology developed in Chapter 4 in \cite{FSV} and we also use the notations in op.cit. For these definitions, we refer to \cite{FSV}. Via Proposition 5.8, Theorem 8.3 (the homotopy invariance) in \cite{FSV} and Theorem \ref{Comp alt}, we have:
\begin{equation}
\begin{split}
&H^{i}(\widetilde{C}^{Alt}_{a - 2b - *}(z_{q.fin}(\mb{A}^{a-b}))(E^{a})) \cong H^{i}(C_{a - 2b - *}(z_{equi}(E^a, \mb{A}^{a-b}, 0))(Spec(k))) \\
\cong &H^{i}(C_{a - 2b - *}(z_{equi}(E^a \times \mb{A}^{a-b}, a))(Spec(k))) \cong H^{i}(C_{a - 2b - *}(z_{equi}(E^a, b))(Spec(k)))\\
\cong & A_{b, 2a-2b-i}(E^a) = H^{BM}_{2a-i}(E^a, \mb{Q}(b)) \cong H^{a+i}(E^a, \mb{Q}(a-b)) \\
= & Hom_{\DM_{gm}(k, \mb{Q})}((M(E))^{\otimes a}, \mb{Q}(a-b)[a+i])
\end{split}
\end{equation}
In the third of the above isomoprhisms, we use the comparison between the Borel-Moore motivic homology (defined there by bivariant cycle cohomology) and motivic cohomology. See this statement before Remark 9.5 in op. cit.
\

Notice that the action of $(\mb{Z}/2)^{a}$ on these groups induced by the action on $E^a$ are compatible under the above isomorphisms. Therefore we have:
$$H^{i}(\widetilde{C}^{Alt , -}_{a - 2b - *}(z_{q.fin}(\mb{A}^{a-b}))(E^{a})) \cong Hom_{\DM_{gm}(k, \mb{Q})}((M_{1}(E))^{\otimes a}, \mb{Q}(a-b)[i]).$$
\end{proof}
\begin{lem} \label{Comp of cycle algebra with motive}
The external product, which is defined on the cohomology groups of the cycle complex $\widetilde{C}^{Alt , -}_{a - 2b - *}(z_{q.fin}(\mb{A}^{a-b}))(E^{a})$,  is compatible with the external product on Hom groups $$Hom_{\DM_{gm}(k, \mb{Q})}((M_{1}(E))^{\otimes a}, \mb{Q}(a-b)[i]),$$ defined in the following way:
\begin{equation}
\begin{split}
&Hom_{\DM_{gm}(k, \mb{Q})}((M_{1}(E))^{\otimes a}, \mb{Q}(a-b)[i]) \otimes Hom_{\DM_{gm}(k, \mb{Q})}((M_{1}(E))^{\otimes c}, \mb{Q}(c-d)[j]) \\
\to
&Hom_{\DM_{gm}(k, \mb{Q})}((M_{1}(E))^{\otimes a} \otimes (M_{1}(E))^{\otimes c}, \mb{Q}(a+c-b-d)[i+j]) \\
\to 
& Hom_{\DM_{gm}(k, \mb{Q})}((M_{1}(E))^{\otimes a+c}, \mb{Q}(a+c-b-d)[i+j]),
\end{split}
\end{equation}
where the first map is taking the external product in $\DM_{gm}(k, \mb{Q})$.
\end{lem}
\begin{proof}  
By Lemma \ref{Comp of cycle},
$Hom_{\DM_{gm}(k, \mb{Q})}((M_{1}(E))^{\otimes a}, \mb{Q}(a-b)[i])$ can be identified as the cohomology group of a  subcomplex $\widetilde{C}^{Alt , -}_{a - 2b - *}(z_{q.fin}(\mb{A}^{a-b}))(E^{a})$ of $C_{a-2b-*}(z_{q.fin}(\mb{A}^{a-b}))(E^a)$. The external product of
$$Hom_{\DM_{gm}(k, \mb{Q})}((M_{1}(E))^{\otimes a}, \mb{Q}(a-b)[i])$$
defined as above is just induced by the product defined in Remark \ref{ext product} on cohomology groups of $$C_{a-2b-*}(z_{q.fin}(\mb{A}^{a-b}))(E^a).$$
On the other hand, notice that the product defined in Remark \ref{ext product} on the cohomology groups of the cycle complex
$$\widetilde{C}^{Alt , -}_{a - 2b - *}(z_{q.fin}(\mb{A}^{a-b}))(E^{a})$$
is given by the external product on the cohomology groups of 
$$C_{a-2b-*}(z_{q.fin}(\mb{A}^{a-b}))(E^a).$$
\end{proof}
From now on, we assume that $E$ is an elliptic curve without CM. The CM case will be discussed in the next section.
\begin{lem} \label{Basic lemma I}
The cohomologies of $\ml{E}^{*}_{a,b}$ are canonically isomorphic to the cohomologies of the following complex of $GL_2$ representations $$\bigoplus_{c+2d =a, c, d \geq 0}Hom_{\DM_{gm}(k, \mb{Q})}(Sym^c M_1(E)(d), \mb{Q}(a-b)[*]) \otimes Sym^{c}\V(d+b-a),$$
where we view its differential maps as zero.
\end{lem}
\begin{proof}
By Lemma \ref{Comp of cycle}, we have the following isomorphism between $GL_2$ representations:
\begin{equation} \nonumber
H^{i}(\ml{E}^{*}_{a,b}) \cong Hom_{\DM_{gm}(k, \mb{Q})}((M_{1}(E))^{\otimes a}, \mb{Q}(a-b)[i]) \otimes_{\mb{Q}[\Sym_a]} \V^{a}(b-a).
\end{equation}
Then by Lemma \ref{Basic decom of elliptic motive}, we know that:
\begin{equation} \nonumber
\begin{split}
&Hom_{\DM_{gm}(k, \mb{Q})}((M_{1}(E))^{\otimes a}, \mb{Q}(a-b)[i]) \otimes_{\mb{Q}[\Sym_a]} \V^{a}(b-a) \\
\cong & Hom_{\DM_{gm}(k, \mb{Q})}(\oplus_{c+2d = a, c,d \geq 0} V_{(c+d,d)} \otimes Sym^c(M_1(E))(d), \mb{Q}(a-b)[i]) \\
& \otimes_{\mb{Q}[\Sym_a]} (\oplus_{e+2f = a, e,f \geq 0} V_{(e+f, f)}\otimes Sym^{e}\V(f+b-a)) \\
\cong & \oplus_{c+2d = e+2f = a, c,d,e,f \geq 0}Hom_{\DM_{gm}(k, \mb{Q})}( Sym^c(M_1(E))(d), \mb{Q}(a-b)[i]) \\
& \otimes (V_{(c+d, d)}^{\vee} \otimes_{\mb{Q}[\Sym_a]} V_{(e+f,f)}) \otimes Sym^{e}\V(f+b-a) \\
\cong &\oplus_{c+2d =a, c, d \geq 0}Hom_{\DM_{gm}(k, \mb{Q})}(Sym^c M_1(E)(d), \mb{Q}(a-b)[i]) \\
&\otimes Sym^{c}\V(d+b-a).
\end{split}
\end{equation}
Notice that given two irreducible representations $V, W$ of a finite group $G$ over $\mb{Q}$, then $V \otimes_{\mb{Q}[G]} W = \mb{Q}$ if $V \cong W$. Otherwise, it's zero. 
\end{proof}

\begin{cor}  \label{0-th vanishing}
If the $0$-vanishing property, defined in Definition \ref{r-th van}, holds for the elliptic curve $E$, then for any $a > 0$, the cohomolgies of $\ml{E}_{2a,a}^{*}$ are all isomorphic to the trivial $GL_{2}$-representation $\mb{Q}$ concentrated in degree zero.
\end{cor}
\begin{proof}
By Lemma \ref{Basic lemma I}, we have:
\begin{equation} 
\begin{split}
&H^{*}(\ml{E}^{*}_{2a, a}) \cong  \bigoplus_{c+2d =2a, c, d \geq 0}Hom_{\DM_{gm}(k, \mb{Q})}(Sym^c M_1(E), \mb{Q}(a-d)[*]) \otimes Sym^{c}\V(d-a).
\end{split}
\end{equation}

From Definition \ref{r-th van}, we know that:
$$Hom_{\DM_{gm}(k, \mb{Q})}(Sym^c M_1(E), \mb{Q}(a-d)[*]) \cong 0 \ \ \ \text{if} \ \ \  c \geq 1.$$
Therefore, we have:
\[
H^n(\ml{E}^{*}_{2a, a}) \cong 
\begin{cases}
0 & \text{if} \ n \neq 0; \\
\mb{Q} & \text{if} \ n = 0,
\end{cases}
\]
where $\mb{Q}$ is the trivial $GL_2$ representation.
\end{proof}

Recall in the previous section, we have defined $\eta$ in equality (\ref{colimit process}). In the next lemma, we want to give a description of $\eta$ under the identification in Lemma \ref{Basic lemma I}. 
\begin{lem} \label{eta}
Via the identification of Lemma \ref{Basic lemma I}, the map: 
$$\eta: \ml{E}^{*}_{a,b} \to \ml{E}^{*}_{a+2,b+1},$$
induced the following map on cohomology groups:
\begin{equation} \nonumber
\begin{split}
& (Hom_{\DM_{gm}(k, \mb{Q})}(Sym^c M_1(E)(d), \mb{Q}(a-b)[*]) \otimes Sym^{c}\V(d+b-a)) \\
& \otimes (Hom_{\DM_{gm}(k, \mb{Q})}(\mb{Q}(1), \mb{Q}(1)) \otimes \mb{Q})\\ 
\to & (Hom_{\DM_{gm}(k, \mb{Q})}(Sym^{c}M_1(E)(d+1), \mb{Q}(a-b+1)[*]) \\
& \otimes Sym^{c}\V(d+b-a).
\end{split}
\end{equation}
Moreover, the map on cohomology groups induces by $\eta$ is a monomorphism in the category of $GL_2$ representations.
\end{lem}
\begin{proof}
By Example \ref{ex for van}, we have a simple description of $\ml{E}^{*}_{2,1}$:
$$H^{*}(\ml{E}^{*}_{2,1}) \cong Hom_{\DM_{gm}(k, \mb{Q})}(\mb{Q}(1), \mb{Q}(1)) \otimes \mb{Q}.$$
Using Lemma \ref{Comp of cycle algebra with motive} and Lemma \ref{Basic lemma I}, we can identify $\eta$ as sending the piece 
$$Hom_{\DM_{gm}(k, \mb{Q})}(Sym^c M_1(E)(d), \mb{Q}(a-b)[*]) \otimes Sym^{c}\V(d+b-a)$$
to
\begin{equation} \nonumber
\begin{split}
Hom_{\DM_{gm}(k, \mb{Q})}(Sym^c M_1(E)(d+1), \mb{Q}(a-b+1)[*]) \otimes Sym^{c}\V(d+b-a)
\end{split}
\end{equation}
By Voevodsky's cancellation theorem in \cite{V10}, on each piece of $\ml{E}^{*}_{a, b}$, $\eta$ is an isomorphism, which implies that $\eta$ is an injection.
\end{proof}

\begin{cor} \label{r-th vanishing}
If an elliptic curve $E$ satisfies the $r$-th vanishing property, defined in Definition \ref{r-th van}, for all positive integer $r$, then all the $H^*(\ml{E}^{*}_{-r})$ are zero. Furthermore, if the elliptic curve $E$ satisfies the $r$-th vanishing property for all non-negative integer $r$, then we have $H^*(\ml{E}^{*}) = H^*(\ml{E}_{ell}^{*})$. 
\end{cor}
\begin{proof}
By Lemma \ref{Basic lemma I}, we have a quasi-isomorphism:
\begin{equation} \nonumber
\begin{split}
&H^{*}(\ml{E}^{*}_{r+2i, r+i}) \cong \bigoplus_{c+2d =r+2i, c, d \geq 0}Hom_{\DM_{gm}(k, \mb{Q})}(Sym^c M_1(E)(d), \mb{Q}(i)[*]) \otimes Sym^c \V(d-i).
\end{split}
\end{equation}
If $E$ satisfies the $r$-th vanishing property for $r \in \mb{Z}_{> 0}$, we have:
\begin{equation} \nonumber
\begin{split}
& Hom_{\DM_{gm}(k, \mb{Q})}(Sym^c M_1(E), \mb{Q}(i-d)[*]) \\
\cong & Hom_{\DM_{gm}(k, \mb{Q})}(Sym^{r+2(i-d)} M_1(E), \mb{Q}(i-d)[*]) \cong 0.
\end{split}
\end{equation}
Therefore, $H^{*}(\ml{E}^{*}_{r+2i, r+i}) \cong 0$ for any $r \in \mb{Z}_{> 0}$ and any $i \in \mb{Z}_{\geq 0}$, which implies that $H^*(\ml{E}^{*}_{-r}) = 0$.
\

Furthermore, if $E$ also satisfies the $0$-th vanishing property, then by Corollary \ref{0-th vanishing}, we know that $H^{*}(\ml{E}_{2a,a}) \cong \mb{Q}$. Also, from Lemma \ref{eta}, we know the connecting map $\eta$ is the identity. Therefore we obtain that $H^{*}(\ml{E}^{*}) = H^*(\ml{E}_{ell}^{*})$.
\end{proof}

\begin{cor} \label{cycle algebra GL_2 identification}
Let $a$ be any integer. In the derived category of $GL_{2}$ representations,  we have the following isomorphisms:
\begin{equation} \nonumber
\begin{split}
&H^*(\ml{E}^{*}_{a}) \cong \bigoplus_{i \geq 0, a\equiv i(2)} Hom_{\DM_{gm}(k, \mb{Q})}(Sym^i M_1(E), \mb{Q}(\frac{a+i}{2})[*]) \otimes Sym^i \V(-\frac{a+i}{2}).
\end{split}
\end{equation}
\end{cor}
\begin{proof}
Using Lemma \ref{Basic lemma I}, we have:
\begin{equation} \nonumber
\begin{split}
&H^*(\ml{E}^{*}_{a,b}) \cong 
\bigoplus_{c+2d =a, c, d \geq 0}Hom_{\DM_{gm}(k, \mb{Q})}(Sym^c M_1(E)(d), \mb{Q}(a-b)[*]) \otimes Sym^{c}\V(d+b-a).
\end{split}
\end{equation}
By Lemma \ref{eta}, the connecting map 
$$\eta: \ml{E}^{*}_{a,b} \to \ml{E}^{*}_{a+2, b+1}$$
sends the summand 
$$Hom_{\DM_{gm}(k, \mb{Q})}(Sym^c M_1(E)(d), \mb{Q}(a-b)[*]) \otimes Sym^{c}\V(d+b-a),$$
of $H^{*}(\ml{E}^{*}_{a,b})$ to the same direct summand in $H^{*}(\ml{E}^{*}_{a+2, b+1})$ by the identity map. Therefore taking the direct limit, we will get the direct sum of all the pieces of the form $$Hom_{\DM_{gm}(k, \mb{Q})}(Sym^c M_1(E)(d), \mb{Q}(a-b)[*]) \otimes Sym^{c}\V(d+b-a).$$
Rewriting the index set, one obtains the desired presentations.
\end{proof}

Next we want to compute the hom-groups between some special dg $\ml{E}_{ell}$-modules. 
\

We let $\ml{T}^{GL_2}_{\ml{E}_{ell}}$ be the full triangulated subcategory of the derived category of dg $\ml{E}_{ell}$-module generated by the dg $\ml{E}_{ell}$-module of the form $\{\ml{E}_{ell} \otimes \V^{-a}(b)[n] \}_{a , b, n \in \mb{Z}}$. Simply denote these elements by $\ml{E}_{ell}\langle a,b\rangle[n]$.
\

For convenience, we use the index $\ml{E}_{ell}$ to denote the hom group in $\ml{T}^{GL_2}_{\ml{E}_{ell}}$ and use $GL_2$ to stand for the derived category of $GL_2$ representations in next lemma.
\begin{equation}
\begin{split}
&Hom_{\ml{E}_{ell}}(\ml{E}_{ell}\langle a,b \rangle[n], \ml{E}_{ell}\langle c,d \rangle[m]) \\
=& Hom_{\ml{E}_{ell}}(\ml{E}_{ell} \otimes \V^{\otimes -a}(b)[n], \ml{E}_{ell} \otimes \V^{\otimes -c}(d)[m]) \\
=& Hom_{GL_{2}}(\V^{\otimes -a}(b), \ml{E}_{ell} \otimes \V^{\otimes -c}(d)[m-n]) \\
=& H^{m-n}(Hom_{GL_{2}}(\mb{Q}, \ml{E}_{ell} \otimes \V^{\otimes -c} \otimes \V^{\otimes a}(-b+d))).
\end{split}
\end{equation}

\begin{lem} \label{fully-faithful}
For $a,b,c,d,i \in \mb{Z}$,  we have:
\begin{equation} \nonumber
\begin{split}
& H^{i}(Hom_{GL_{2}}(\mb{Q}, \ml{E}_{ell} \otimes \V^{\otimes -c} \otimes \V^{\otimes a}(-b+d))) \\
\cong & Hom_{\DM_{gm}(k, \mb{Q})}((M_{1}(E))^{\otimes -a}(b), (M_{1}(E))^{\otimes -c}(d)[i]).
\end{split}
\end{equation}
\end{lem}
Convention: For simplicity, we denote the multiplicity of $Sym^a \V \otimes det^b$ in $\V^{\otimes n}$ by $C_{a, b}$ and denote the multiplicity of $Sym^{a+b-2i} \V \otimes det^i$ in $Sym^a \V \otimes Sym^b \V$ by $D^{i}_{a, b}$.
\begin{proof}
The isomorphism $\V^{\vee} \cong \V(-1)$ gives us the isomorphism $\V^{-a} \cong \V^{a}(-a)$ in the category of $GL_2$-representations. Similarly using $M_1(E)^{\vee} \cong M_1(E)(-1)$ gives us an isomorphism between geometric motives $M_1(E)^{-a} \cong M_1^{\otimes a}(E)(-a)$. Without loss of generality, we can assume $a, c \geq 0$. By Voevodsky's Cancellation theorem, we can also assume that $b = 0$.  For simplicity, we only prove the case $d = 0$. 
\

$\V^{\otimes -c} \otimes \V^{\otimes a}$ is the direct sum of $Sym^{a - 2n}\V(n) \otimes Sym^{- c - 2m}\V(m)$, where $0 \leq 2n \leq a, c + 2m \leq 0, c + m \geq 0$, with multiplicities $C_{a-2n,n} \times C_{-c-2m,c+m}$. 
\

Furthermore, we can decompose $\V^{\otimes -c} \otimes \V^{\otimes a}$ as the direct sum of irreducible $GL_{2}$ representations of the form $Sym^{a-2n-c-2m-2l}\V(m+n+l)$, where the index set $\mu$ satisfies $0 \leq 2n \leq a, c + 2m \leq 0, c + m \geq 0, 0 \leq 2l \leq a-c-2(m+n)$, with multiplicity $C_{a-2n,n} \times C_{-c-2m,c+m} \times D_{a-2n, -c-2m}^{l}$. From this decomposition, we get $m+c+n+l \geq 0$, which implies that $a-2m-c-2n-2l \leq a+c$. 
\

For each irreducible representation $Sym^{a-2n-c-2m-2l}\V(m+n+l)$, we have:
\begin{equation} \label{computation of full}
\begin{split}
& H^{i}(Hom_{GL_{2}}(\mb{Q}, \ml{E}^{*}_{ell} \otimes Sym^{a-2n-c-2m-2l}\V(m+n+l))) \\
\cong & H^{i}(Hom_{GL_{2}}(\mb{Q}, \ml{E}^{*}_{c-a} \otimes Sym^{a-2n-c-2m-2l}\V(m+n+l))) \\
\cong & H^{i}(Hom_{GL_{2}}((Sym^{a-2n-c-2m-2l}\V(m+n+l))^{*}, \ml{E}^{*}_{c-a})) \\
\cong & H^{i}(Hom_{GL_{2}}(Sym^{a-2n-c-2m-2l}\V(m+n+l-a+c), \ml{E}^{*}_{c-a})) \\
\cong & Hom_{\DM_{gm}(k, \mb{Q})}(Sym^{a-2n-c-2m-2l}M_{1}(E), \mb{Q}(a-c-m-n-l)[i]),
\end{split}
\end{equation}
For the last isomorphism, we use Corollary \ref{cycle algebra GL_2 identification}. 
\

On the other hand, let us compute the hom-groups between motives. 
\begin{equation}
\begin{split}
& Hom_{\DM_{gm}(k, \mb{Q})}((M_{1}(E))^{\otimes -a}, (M_{1}(E))^{\otimes -c}[i]) \\
\cong & Hom_{\DM_{gm}(k, \mb{Q})}(M_{1}(E)^{\otimes a} \otimes M_{1}(E)^{\otimes c}, \mb{Q}(a)[i]) \\
\cong & \oplus_{0 \leq 2s \leq a, 0 \leq 2t \leq c}(C_{a-2s,s} \times C_{c-2t,t}) \\
&\cdot Hom_{\DM_{gm}(k, \mb{Q})}(Sym^{a-2s}M_{1}(E)(s) \otimes Sym^{c-2t}M_{1}(E)(t), \mb{Q}(a)[i]) \\
\cong & \oplus_{0 \leq 2s \leq a, 0 \leq 2t \leq c, 0 \leq 2r \leq a+c-2s-2t}(C_{a-2s,s} \times C_{c-2t,t} \times D^{r}_{a-2s, c-2t}) \\
& \cdot Hom_{\DM_{gm}(k, \mb{Q})}(Sym^{a+c-2s-2t-2r}M_{1}(E)(r), \mb{Q}(a-s-t)[i])
\end{split}
\end{equation}
Rewrite the index set, and let
$s =n, t = c+m, r = l.$
Then this index set is the same as $\mu$. Notice that the multiplicities of the term 
$$Hom_{\DM_{gm}(k, \mb{Q})}(Sym^{a-2n-c-2m-2l}M_{1}(E), \mb{Q}(a-c-m-n-l)[i])$$
in 
$$H^{i}(Hom_{GL_{2}}(\mb{Q}, \ml{E}^{*}_{ell} \otimes Sym^{a-2n-c-2m-2l}\V(m+n+l)))$$
and
$$Hom_{\DM_{gm}(k, \mb{Q})}((M_{1}(E))^{\otimes -a}, (M_{1}(E))^{\otimes -c}[i])$$
are the same. Both are $C_{a-2n,n} \times C_{-c-2m,c+m} \times D_{a-2n, -c-2m}^{l}$.
\end{proof}

\section{Computations for the CM case} \label{sec-compcm}
We let $E$ be an elliptic curve with CM. The computation results are parallel to the previous section. Recall that in this case every representation and cycles complexes $\ml{E}^{*}_{a, b}$ are considered over $\mb{K}$ and we have a decomposition of the standard representation of $GL_2: \V = V \oplus \bar{V}$. Because the proof of the results in this section are similar as before and short the length of this paper, we omit the proof and only state the results. Notice that $M_1(E) = M \oplus \bar{M}$ in $\DM_{gm}(k, \mb{K})$. For simplicity, the exponent $e$ on $M$, $\bar{M}$, $M$, $\bar{M}$ means the $e$-th tensor power of these objects.
\begin{lem} \label{Basic lemma II}
(Compare with Lemma \ref{Basic lemma I}) The cohomologies of $\ml{E}^{*}_{a,b}$ are canonically isomorphic to the cohomologies of the following complex of $T_{\mb{K}}$-representations $$\bigoplus_{e+f+2d =a, d, e, f \geq 0}Hom_{\DM_{gm}(k, \mb{K})}(M^e \otimes \bar{M}^f, \mb{K}(a-b-d)[*]) \otimes (V^e \otimes \bar{V}^f)(d+b-a),$$
where we view its differential maps as zero.
\end{lem}
\begin{cor}  \label{0-th vanishing cm}
If the $0$-vanishing property holds for an elliptic curve $E$, then for any $a > 0$, the cohomolgies of $\ml{E}_{2a,a}^{*}$ are all isomorphic to the direct sum of $a+1$ trivial $T_{\mb{K}}$-representations $\mb{K}$ concentrated in degree zero. Therefore, for any $a > 0$, the cohomolgies of $\ml{A}_{2a,a}^{*}$ are all isomorphic to the trivial $T_{\mb{K}}$-representations $\mb{K}$ concentrated in degree zero.
\end{cor}
\begin{lem}\label{Basic lemma III}
The cohomologies of $\ml{A}^{*}_{a,b}$ are canonically isomorphic to the cohomologies of the following complex of $T_{\mb{K}}$-representations $$\bigoplus_{e+f=a, e, f \geq 0}Hom_{\DM_{gm}(k, \mb{K})}(M^e \otimes \bar{M}^f, \mb{K}(a-b)[*]) \otimes (V^e \otimes \bar{V}^f)(b-a),$$
where we view its differential maps as zero.
\end{lem}
\begin{lem} \label{beta}
Via the identification of Lemma \ref{Basic lemma III}, the map: 
$$\bar{\eta}: \ml{A}^{*}_{a,b} \to \ml{A}^{*}_{a+2,b+1},$$
induces the following map on cohomology groups:
\begin{equation} \nonumber
\begin{split}
& (\oplus_{e+f=a, e, f \geq 0}Hom_{\DM_{gm}(k, \mb{K})}(M^e \otimes \bar{M}^f, \mb{K}(a-b)[*]) \otimes (V^e \otimes \bar{V}^f)(b-a)) \\
& \otimes (Hom_{\DM_{gm}(k, \mb{K})}(M \otimes \bar{M}, \mb{K}(1)) \otimes (V \otimes \bar{V} \otimes det^{-1}))\\ 
\to & (\oplus_{e+f=a+2, e, f \geq 0}Hom_{\DM_{gm}(k, \mb{K})}(M^e \otimes \bar{M}^f, \mb{K}(a-b+1)[*]) \otimes (V^e \otimes \bar{V}^f)(b-a-1)).
\end{split}
\end{equation}
Moreover, the map on cohomology groups induces by $\eta$ is a monomorphism in the category of $T_{\mb{K}}$ representations.
\end{lem}
\begin{cor} \label{r-th vanishing cm}
(Compare with Corollary \ref{r-th vanishing}) If an elliptic curve $E$ satisfies the $r$-th vanishing property for all positive integer $r$, then all the $H^*(\ml{E}^{*}_{-r})$ are zero. Furthermore, if the elliptic curve $E$ satisfies the $r$-th vanishing property for all non-negative integer $r$, then we have $H^*(\ml{E}^{*}) = H^*(\ml{E}_{ell}^{*})$. 
\end{cor}
\begin{cor} \label{cycle algebra T_K identification}
(Compare with Corollary \ref{cycle algebra GL_2 identification}) Let $a$ be any integer. In the derived category of $T_{\mb{K}}$ representations,  we have the following isomorphisms:
\begin{equation} \nonumber
\begin{split}
H^*(\ml{E}^{*}_{a}) \cong 
& \bigoplus_{i \geq 0, a\equiv i(2)} Hom_{\DM_{gm}(k, \mb{K})}(M^i, \mb{K}(\frac{a+i}{2})[*]) \otimes V^i(-\frac{a+i}{2}) \\
& \oplus  \bigoplus_{j > 0, a\equiv j(2)} Hom_{\DM_{gm}(k, \mb{K})}(\bar{M}^j, \mb{K}(\frac{a+j}{2})[*]) \otimes \bar{V}^i(-\frac{a+j}{2}).
\end{split}
\end{equation}
\end{cor}
\begin{lem} \label{fully-faithful cm}
(Compare with Lemma \ref{fully-faithful}) For $a,b,c,d,i \in \mb{Z}$,  we have:
\begin{equation} \nonumber
\begin{split}
& H^{i}(Hom_{T_{\mb{K}}}(\mb{K}, \ml{E}_{ell} \otimes \V^{\otimes -c} \otimes \V^{\otimes a}(-b+d))) \\
\cong & Hom_{\DM_{gm}(k, \mb{K})}((M_{1}(E))^{\otimes -a}(b), (M_{1}(E))^{\otimes -c}(d)[i]).
\end{split}
\end{equation}
\end{lem}

\section{The motivic version of cycle cdgas in $\DM(k)$} \label{cycle alg}
\begin{df} \label{sheaf version}
Let $a, b$ be integers such that $a \geq b, a \geq 0$. 
\begin{itemize}
\item
For $i \in \mb{Z}$ and $T \in \Sm$ and $E$ an elliptic curve without CM, we define:
$$\mf{E}_{a, b}^{i}(T) = \widetilde{C}^{Alt , -}_{a - 2b - i}(z_{q.fin}(\mb{A}^{a - b}))(E^{a} \times T) \otimes_{\mb{Q}[\Sigma_{a}]}  \V^{\otimes a}(b-a).$$
\item
For $i \in \mb{Z}$ and $T \in \Sm$ and $E$ an elliptic curve with CM, we define:
$$\mf{E}_{a, b}^{i}(T) = \widetilde{C}^{Alt , -}_{a - 2b - i}(z_{q.fin}(\mb{A}^{a - b}))(E^{a} \times T) \otimes_{B_{n, \mb{K}}}  \V^{\otimes a}(b-a).$$
\end{itemize}
Then $\mf{E}_{a, b}^{i}$\footnote{This could be thought as the ``motivic version'' of the cycle algebra $\ml{E}_{a, b}^i$.} is a $\R$(resp $\mathbf{Rep}_{T_{\mb{K}}}$)-valued Nisnevich sheaf with transfers.
\end{df}
\begin{rmk} \label{com-trick}
From the definition, we have $\mf{E}^{*}_{a, b}(k) = \ml{E}^{*}_{a, b}$. In fact, by computations similar in Section \ref{sec-comp}, one can get the following isomorphism in $\DM_{gm}(k, \mb{Q})$:
$$\mf{E}_{a, b}^{*} \cong \Hom(M_1(E)^{\otimes a}, \mb{Q}(a-b)) \otimes_{\mb{Q}[\Sigma_{a}]} \V^{\otimes a}(b-a).$$
Here $\Hom$ is defined in Remark 14.12 in \cite{MVW}.
In CM case, we have the isomorphism in $\DM_{gm}(k, \mb{K})$:
$\mf{E}_{a, b}^{*} \cong \Hom(M_1(E)^{\otimes a}, \mb{K}(a-b)) \otimes_{B_{n, \mb{K}}} \V^{\otimes a}(b-a).$
\end{rmk}
\begin{rmk}
In the non-CM case, $\{\mf{E}^{*}_{a, b}\}$ is a cdga over $GL_2$ object in $C(\Shk)_{\mb{Q}}$. More precisely, for $S, T \in \Sm$, the external product of correspondences gives the following product map:
\begin{equation}
\begin{split}
&C_{a - 2b - i}(z_{q.fin}(\mb{A}^{a - b}))(E^{a} \times S) \otimes C_{c - 2d - i}(z_{q.fin}(\mb{A}^{c - d}))(E^{c} \times S) \\
\longrightarrow & C_{a+c - 2b- 2d - i}(z_{q.fin}(\mb{A}^{a +c - b - d}))(E^{a} \times S \times E^{c} \times T)
\end{split}
\end{equation}
Taking the alternating projection with respect to the component $\square$, $-$ part with respect to the component $E$ and symmetric projection with respect to the component $\mb{A}$ yields:
\begin{equation}
\begin{split}
&\widetilde{C}^{Alt , -}_{a-2b-i}(z_{q.fin}(\mb{A}^{a - b}))(E^{a} \times S) \otimes \widetilde{C}^{Alt , -}_{c - 2d - i}(z_{q.fin}(\mb{A}^{c - d}))(E^{c} \times S) \\
\longrightarrow & \widetilde{C}^{Alt , -}_{a+c - 2b- 2d - i}(z_{q.fin}(\mb{A}^{a +c - b - d}))(E^{a} \times S \times E^{c} \times T)
\end{split}
\end{equation}
Then we get the map as in (\ref{mul-str}):
\begin{equation}
\cdot: \mf{E}^*_{a,b} \otimes \mf{E}^*_{c,d} \to \mf{E}^*_{a+c, b+d}.
\end{equation}
As before, one may check this map is associative and graded commutative. In the CM case, similarly we can show that $\{\mf{E}^{*}_{a, b}\}$ is a cdga over $T_{\mb{K}}$ object in $C(\Shk)_{\mb{K}}$.
\begin{rmk}
When $E$ is an elliptic curve without CM, one important observation is:
\begin{equation} \nonumber
\begin{split} 
& H^*(\mf{E}_{2, 1}^{*}) \cong \Hom(M_1(E)^{\otimes 2}, \mb{Q}(1)) \otimes_{\mb{Q}[\Sigma_{2}]}  \V^{\otimes 2}(-1) \\
\cong & \Hom(Sym^2(M_1(E)), \mb{Q}(1)) \otimes  Sym^2\V(-1) \oplus \mb{Q} \cong \mb{Q} \in DM^{eff}(k, \mb{Q}).
\end{split}
\end{equation}
This computation relies on Proposition 13.7 in \cite{MVW} and the fact that: for any field $k^{'}$,
$$\Hom(Sym^2(M_1(E)), \mb{Q}(1))(k^{'}) \cong 0,$$
whose proof is the same in Example \ref{ex for van}. If $E$ is an elliptic curve with CM, then $$H^*(\mf{E}_{2, 1}^{*}) \cong \mb{K} \otimes (V \otimes \bar{V}(-1)) \cong \mb{K}.$$
\end{rmk}
Similarly using the multiplicative structure, we have: 
$$\eta: \mf{E}^*_{a,b} \to \mf{E}^*_{a+2,b+1}.$$
Furthermore, if $E$ is with CM, we have a map as explained in Remark \ref{existence of beta}:
$$\tau: \mf{E}^*_{a, b} \to \mf{E}^*_{a+2, b+1}.$$
We now define $\mf{E}^{*}$ as in Definition \ref{df of elliptic-alg}.
\begin{df}  \label{df of motivic elliptic-alg}
Given $a \in \mb{Z}$, 
\begin{itemize}
\item
for an elliptic curve without CM, we define: $$\mf{E}_{a}^{*} =   \underrightarrow{lim}_{i \geq -a} \  \mf{E}^{*}_{-a+2i, -a+i},$$
where the colimits are taken over the map $\eta$.
\item
for an elliptic curve with CM, we define:$$\mf{E}_{a}^{*} =   \underrightarrow{lim}_{i \geq -a} \  \mf{E}^{*}_{-a+2i, -a+i},$$
where the colimits are taken over the map $\tau$.
\end{itemize}
Then we denote:
$$\mf{E}^{*} = \mb{Q} \oplus \bigoplus_{a \geq 1}\mf{E}_{a}^{*}.$$
and
$$\mf{E}_{ell}^{*} = \bigoplus_{a \in \mb{Z}} \mf{E}^{*}_{a}.$$   
\end{df}
\end{rmk}
\begin{prop} \label{cdga object in DM}
\begin{itemize}
\item
If $E$ is an elliptic curve without CM,  then $\mf{E}^{*}$ and $\mf{E}_{ell}^{*}$ are commutative monoids in the category of complexes of $\R$-valued Nisnevich sheaves with transfers.
\item
If $E$ is an elliptic curve with CM,  then $\mf{E}^{*}$ and $\mf{E}_{ell}^{*}$ are commutative monoids in the category of complexes of $T_{\mb{K}} (= \mb{G}_{m, \mb{K}} \times \mb{G}_{m, \mb{K}})$-valued Nisnevich sheaves with transfers.
\end{itemize}
\end{prop}
\begin{proof}
The proof can be found in Section 4.3 of \cite{L}.
\end{proof}
\begin{rmk} \label{important for ml}
Following the same proofs as Lemma \ref{Basic lemma I}, Lemma \ref{eta} and Corollary \ref{cycle algebra GL_2 identification}, we obtain the following properties of $\mf{E}^{*}_{a,b}$ for $E$ an elliptic curve without CM. 
\begin{enumerate}
\item[(a).]
The cohomologies of $\mf{E}^{*}_{a,b}$ are canonically isomorphic to the cohomologies of the following complex of $GL_2$ representations $$\bigoplus_{c+2d =a, c, d \geq 0}\Hom(Sym^c M_1(E)(d), \mb{Q}(a-b)[*]) \otimes Sym^{c}\V(d+b-a),$$
where we view the differentials as zero.
\item[(b).]
Via the identification of property(a), the map: 
$$\eta: \mf{E}^{*}_{a,b} \to \mf{E}^{*}_{a+2,b+1},$$
is compatible with the following map:
\begin{equation} \nonumber
\begin{split}
& (\Hom(Sym^c M_1(E)(d), \mb{Q}(a-b)[*]) \otimes Sym^{c}\V(d+b-a))  \otimes (\Hom(\mb{Q}(1), \mb{Q}(1)) \otimes \mb{Q})\\ 
\to & (\Hom(Sym^{c}M_1(E)(d+1), \mb{Q}(a-b+1)[*])  \otimes Sym^{c}\V(d+b-a).
\end{split}
\end{equation}
Moreover, the maps on cohomologies induced by $\eta$ are injective in the category of $GL_2$ representations.
\item[(c).]
Let $a$ be any non negative integer. In the derived category of $GL_{2}$-representations,  we have the following isomorphisms:
\begin{equation} \nonumber
\begin{split}
&H^*(\mf{E}^{*}_{a}) \cong 
 \bigoplus_{i \geq 0, a \equiv i(2)} \Hom(Sym^i M_1(E), \mb{Q}(\frac{a+i}{2})[*]) \otimes Sym^i \V(-\frac{a+i}{2}).
\end{split}
\end{equation}
\end{enumerate}
For $E$ an elliptic curve with CM, the results in Section \ref{sec-compcm} hold. In particular, we let $a$ be any non negative integer. In the derived category of $T_{\mb{K}}$-representations,  we have the following isomorphisms:
\begin{equation} \nonumber
\begin{split}
H^*(\ml{E}^{*}_{a}) \cong 
& \bigoplus_{i \geq 0, a\equiv i(2)} \Hom(M^i, \mb{K}(\frac{a+i}{2})[*]) \otimes V^i(-\frac{a+i}{2}) \\
& \oplus  \bigoplus_{j > 0, a\equiv j(2)} \Hom(\bar{M}^j, \mb{K}(\frac{a+j}{2})[*]) \otimes \bar{V}^i(-\frac{a+j}{2}).
\end{split}
\end{equation}
\end{rmk}
\

\section{DG modules and motives for an elliptic curve}
\label{DG modules and motives for an elliptic curve}
In this section, we want to connect the dg $\ml{E}^*_{ell}$-module with Voevodsky's geometric motives. Let us first explain the case of elliptic curves without CM in details.
\

Fix $r \in \mb{Z}_{\geq 0}$. Given $M \in \CM^{GL_2}_{\ml{E}_{ell}^{*}}$, we define its Adams graded $r$ summand as:
$$M(r) = Hom_{GL_2}(det^{\otimes -r}, \mf{E}_{ell}^{*} \otimes_{\ml{E}_{ell}^{*}}  M[2r]).$$
Here $[2r]$ means the shift of the complex. $Hom_{GL_2}(\cdot, \cdot)$ is the usual hom complex in $C(\R)$.
In fact, this defines a dg functor:
$$\ml{M}(r)^{dg}: \CM^{GL_2}_{\ml{E}_{ell}^{*}} \to C(\Shk)$$
and also an exact functor :
$$\ml{M}(r): \KCM^{GL_2}_{\ml{E}_{ell}^{*}} \to D(\Shk).$$
\begin{df}
Let $T^{tr}$ be the presheaf with transfers:
$$T^{tr} = coker(\mb{Q}_{tr}(Spec(k)) \xrightarrow{i_{\infty *}} \mb{Q}_{tr}(\mb{P}^1)),$$
where $i_{\infty}$ is the inclusion of $\infty$ into $\mb{P}^1$.
\end{df}
In fact, $T^{tr}$ is a Nisnevich sheaf with transfers.
\begin{lem}
We have a natural injectve map in $C(\Shk)$:    %%%%% need to think shift.
$$T^{tr} \to H^{0}(GL_2, \mf{E}_{ell}^{*} \otimes det[2]) .$$
\end{lem}
\begin{proof}
By the definition of $\mf{E}^{*}$, its $det^{-1}$ isotypical part is given by
$$\lim_{i \geq 0} \mf{E}^{*}_{2i,i-1}.$$ Notice that $$\mf{E}^{*}_{0, -1} \cong \widetilde{C}^{Alt , -}_{2-*}(z_{q.fin}(\mb{A}^{1})) \cong T^{tr}[-2].$$ 
So there is a natural injective map:
$$T^{tr} \to H^{0}(GL_2, \mf{E}_{ell}^{*} \otimes det[2]).$$
\end{proof}
For $M \in \CM^{GL_2}_{\ml{E}_{ell}^{*}}$, from the above lemma, we have the following composition of maps:
\begin{equation}
\begin{split}
&T^{tr} \ot \ml{M}(r)^{dg}(M) = T^{tr} \ot Hom_{GL_2}(det^{\otimes -r}, \mf{E}_{ell} \otimes_{\ml{E}_{ell}}  M[2r]) \\
\longrightarrow & Hom_{GL_2}(det^{-1}, \mf{E}_{ell}[2]) \ot Hom_{GL_2}(det^{\otimes -r}, \mf{E}_{ell} \otimes_{\ml{E}_{ell}}  M[2r]) \\
\longrightarrow & Hom_{GL_2}(det^{\otimes -r-1}, \mf{E}_{ell} \otimes \mf{E}_{ell} \otimes_{\ml{E}_{ell}}  M[2r+2]) \\
\longrightarrow & Hom_{GL_2}(det^{\otimes -r-1}, \mf{E}_{ell} \otimes_{\ml{E}_{ell}}  M[2r+2]) = \ml{M}(r+1)^{dg}(M).
\end{split}
\end{equation}
For the last arrow, we use the multiplicative structure of  $\ml{E}_{ell}$. Denote the composition of these maps by $\epsilon^{*}_{r}(M)$.
\

In order to construct a functor from the homotopy category of cell modules to the category of motives, we need to use Voevodsky's big category of motives $\DM(k, \mb{Q})$, which is defined by the symmetric spectra. Roughly speaking, one needs to define a model category $\Spt^{\Sym}_{T^{tr}}(k, \mb{Q})$ of symmetric $T^{tr}$ spectra in $C(\Shk)$ with ``a suitable model structure'', and then $\DM(k, \mb{Q})$ is defined to be the homotopy category of $\Spt^{\Sym}_{T^{tr}}(k, \mb{Q})$. For this approach, we refer to section 3.2, 3.3 and 3.4 in \cite{L}.
\

Then sending $M \in \CM^{GL_2}_{\ml{E}_{ell}^{*}}$ to the sequence:
$$\ml{M}^{dg}_{*}(M) = (\ml{M}^{dg}(0)(M), \ml{M}^{dg}(1)(M), \cdots)$$
with the bonding map $\epsilon^{*}_{r}(M)$ defines a dg functor:
$$\ml{M}^{dg}_{*}: \CM^{GL_2}_{\ml{E}_{ell}^{*}} \to \Spt^{\Sym}_{T^{tr}}(k, \mb{Q}),$$
and also an exact functor on their homotopy categories
$$\ml{M}_{*}: \KCM^{GL_2}_{\ml{E}_{ell}^{*}} \to \DM(k, \mb{Q}).$$
Here the $n$-th term in the spectrum is equipped with a trivial $\Sigma_n$-action.
\begin{lem}   \label{comp-Tate}
We have the following isomorphisms  in $\DM(k, \mb{Q})$:
\begin{enumerate}
\item
$\ml{M}(r)(\ml{E}_{ell})  \cong \mb{Q}(r)[2r].$
\item
Given $a, b \in \mb{Z}$, for any $r \in \mb{Z}$ such that $b+r \geq 0$, we have: 
$$\ml{M}(r)(\ml{E}_{ell} \otimes \V^{\otimes a}(b))  \cong M_1(E)^{\otimes a}(b+r)[2r].$$
\end{enumerate}
\end{lem}
\begin{proof}
Because $\ml{M}(r)(\ml{E}_{ell}) = Hom_{GL_2}(det^{\otimes -r}, \mf{E}_{ell}[2r])$, the only non-trivial part is coming from weight $-2r$(or Adams degree $2r$) part in $\mf{E}_{ell}$. Using Remark \ref{important for ml}, we have the following quasi-isomorphism:
\begin{equation} \nonumber
\begin{split}
\mf{E}^{*}_{2r} \cong \bigoplus_{i \geq 0} \Hom(Sym^i M_1(E), \mb{Q}(\frac{2r+i}{2})) \otimes Sym^i \V(-\frac{2r+i}{2}).
\end{split}
\end{equation}
Then by definition of $\ml{M}(r)(\ml{E}_{ell}) $, we have:
\begin{equation} \nonumber
\begin{split}
& \ml{M}(r)(\ml{E}_{ell}) = Hom_{GL_2}(det^{\otimes -r}, \mf{E}_{ell}[2r]) \\
\cong & Hom_{GL_2}(det^{\otimes -r}, \bigoplus_{i \geq 0} \Hom(Sym^i M_1(E), \mb{Q}(\frac{2r+i}{2})) \\
& \otimes Sym^i \V(-\frac{2r+i}{2})[2r]) \\
\cong & Hom_{GL_2}(det^{\otimes -r}, \Hom(\mb{Q}, \mb{Q}(r)) \otimes det^{\otimes -r}[2r]) \\
\cong & \ml{H}om_{\DM_{gm}(k, \mb{Q})}(\mb{Q}, \mb{Q}(r))[2r] \cong \mb{Q}(r)[2r].
\end{split}
\end{equation}
For the second isomorphism,  we need to compute weight $-2r-1$ part in $\mf{E}_{ell}$. We have:
\begin{equation} \nonumber
\begin{split}
\mf{E}^{*}_{2r+1} \cong \bigoplus_{i \geq 0} \Hom(Sym^i M_1(E), \mb{Q}(\frac{2r+1+i}{2})) \otimes Sym^i \V(-\frac{2r+1+i}{2}).
\end{split}
\end{equation}
Therefore:
\begin{equation} \nonumber
\begin{split}
& \ml{M}(r)(\ml{E}_{ell} \otimes \V^{\otimes a}(b)) = Hom_{GL_2}(det^{\otimes -r}, \mf{E}_{ell} \otimes \V^{\otimes a}(b)[2r]) \\
\cong & Hom_{GL_2}(\V^{\otimes -a} \otimes det^{\otimes -b-r}, \mf{E}_{ell}[2r]) \\
\cong& Hom_{GL_2}(\V^{\otimes -a} \otimes det^{\otimes -b-r}, \mf{E}_{2b+2r+a}[2r])\\
\cong & Hom_{GL_2}(\V^{\otimes -a} \otimes det^{\otimes -b-r}, \bigoplus_{i \geq 0} \Hom(Sym^i M_1(E), \mb{Q}(\frac{2b+2r+a+i}{2})) \\
& \otimes Sym^i \V(-\frac{2b+2r+a+i}{2})[2r]) \\
\cong & Hom_{GL_2}(\bigoplus_{0 \leq j \leq a}C_j \otimes Sym^j\V(-\frac{2b+2r+a+j}{2}), \bigoplus_{i \geq 0} \Hom(Sym^i M_1(E), \\
&\mb{Q}(\frac{2b+2r+a+i}{2})) \otimes Sym^i \V(-\frac{2b+2r+a+i}{2})[2r]) \\
\cong & \ml{H}om_{\DM_{gm}(k, \mb{Q})}(\bigoplus_{0 \leq j \leq a} C_j \otimes Sym^j M_1(E), \mb{Q}(\frac{2b+2r+a+j}{2}))[2r]\\
\cong & M_{1}(E)^{\otimes a}(b+r)[2r].
\end{split}
\end{equation}
Here $C_j$ is the multiplicity of $Sym^j\V(-\frac{2b+2r+a+j}{2})$ in $\V^{\otimes -a} \otimes det^{\otimes -b-r}$.
\end{proof}

%%%%%%%%%%%%%%% fully faithful on generator

\begin{lem} \label{Vab generate}
$\{\ml{E}_{ell} \otimes \V^{a}(b) | a, b \in \mb{Z} \}$  generate $\ml{D}^{GL_2}_{\ml{E}_{ell}}$.
\end{lem}
\begin{proof}
Let $M \in \ml{D}^{GL_2}_{\ml{E}_{ell}}$ be a dg module satisfying
$$Hom_{\ml{D}^{GL_2}_{\ml{E}_{ell}} }(\ml{E}_{ell} \otimes \V^{a}(b), M[i]) \cong 0$$
for any $\V^{a}(b) \in \R$ and $a, b, i \in \mb{Z}$.
Without loss of generality, we assume that $M$ is a cell module. Using Remark 6.5 in \cite{C}, we obtain that:
\[
Hom_{\ml{D}^{GL_2}_{\ml{E}_{ell}} }(\ml{E}_{ell} \otimes \V^{a}(b), M[i]) \cong H^i(Hom_{GL_2}(\V^{a}(b), M))\cong 0,
\] 
which implies that $M$ is quasi-isomorphic to $0$ as a complex of $GL_{2}$ representations. 
\end{proof}
%%%%%%%%%%%%% 
%%%%%%%%%%%%%%%%%%%%%%%
\begin{cor} \label{classically generates}
$\{\ml{E}_{ell} \otimes \V^{a}(b) | a, b\in \mb{Z} \}$  classically generates $(\ml{D}^{GL_2}_{\ml{E}_{ell}})^c$.
\end{cor}
\begin{proof}
First we want to show that $\ml{E}_{ell} \otimes \V^{a}(b)$ is a compact object in $\ml{D}^{GL_2}_{\ml{E}_{ell}}$ for any $a, b \in \mb{Z}$. Let $\{M_i \}_{i \in I}$ be a family of cell $A$-modules. By Remark 6.5 in \cite{C}, we have:
\begin{equation} \nonumber
\begin{split}
& Hom_{\ml{D}^{GL_2}_{\ml{E}_{ell}}}(\ml{E}_{ell} \otimes \V^{a}(b), \bigoplus_{i \in I}M_i) = Hom_{\ml{KCM}^{GL_2}_{\ml{E}_{ell}}}(\ml{E}_{ell} \otimes \V^{a}(b), \bigoplus_{i \in I}M_i) \\
\cong & Hom_{\ml{KCM}^{GL_2}_{\mb{Q}}}( \V^{a}(b), \bigoplus_{i \in I}M_i) \cong \bigoplus_{i \in I}Hom_{\ml{KCM}^{GL_2}_{\mb{Q}}}( \V^{a}(b), M_i) \\
\cong & \bigoplus_{i \in I}Hom_{\ml{KCM}^{GL_2}_{\ml{E}_{ell}}}(\ml{E}_{ell} \otimes \V^{a}(b), M_i),
\end{split}
\end{equation}
which implies that $\ml{E}_{ell} \otimes \V^{a}(b)$ is compact. Here we use that $\V^{a}(b)$ is a compact object in $\ml{D}^{GL_2}_{\mb{Q}}$.  
\

Together with Lemma \ref{Vab generate}, we know that $\ml{D}^{GL_2}_{\ml{E}_{ell}}$, as a compactly generated triangulated category, is generated by $\{\ml{E}_{ell} \otimes \V^{a}(b) | a, b\in \mb{Z} \}$. Then using a result of Neeman in \cite{Nee}, we know that $\{\ml{E}_{ell} \otimes \V^{a}(b) | a, b\in \mb{Z} \}$  classically generate $(\ml{D}^{GL_2}_{\ml{E}_{ell}})^c$.
\end{proof}
\begin{rmk}
Recall in Remark 5.9 in \cite{C}, we have:
$$(\KCM^{GL_2, f}_{\ml{E}_{ell}})^{\natural} \subset \ml{KFCM}^{GL_2}_{\ml{E}_{ell}} \subset (\ml{D}^{GL_2}_{\ml{E}_{ell}})^c.$$
Using Corollary \ref{classically generates}, we know that $(\KCM^{GL_2, f}_{\ml{E}_{ell}})^{\natural} \cong (\ml{D}^{GL_2}_{\ml{E}_{ell}})^c$. Therefore, we have:
$$(\KCM^{GL_2, f}_{\ml{E}_{ell}})^{\natural} \cong \ml{KFCM}^{GL_2}_A \cong (\ml{D}^{GL_2}_{\ml{E}_{ell}})^c.$$
\end{rmk}

%%%%%%%%%% lax tensor
\begin{lem}   \label{laxtensor}
The restriction of $\ml{M}$ to $\KCM^{GL_2}_{\ml{E}_{ell}}$ is a lax tensor functor.
\end{lem}
\begin{proof}
Given $M, N \in \KCM^{GL_2}_{\ml{E}_{ell}}$, we have the following maps:
\begin{equation} \nonumber
\begin{split}
&(\mf{E}_{ell} \otimes_{\ml{E}_{ell}} M) \otimes^{tr} (\mf{E}_{ell} \otimes_{\ml{E}_{ell}} N) 
\longrightarrow  (\mf{E}_{ell} \otimes^{tr} \mf{E}_{ell}) \otimes_{\ml{E}_{ell}} (M \otimes_{\ml{E}_{ell}} N) \\
& \longrightarrow \mf{E}_{ell} \otimes_{\ml{E}_{ell}} (M \otimes_{\ml{E}_{ell}} N),
\end{split}
\end{equation}
where the last map is obtained by using the multiplicative structure of $\mf{E}_{ell}$ as a cdga over $GL_2$ in $\DM(k, \mb{Q})$ (Proposition \ref{cdga object in DM}). On the corresponding Adams graded summand, this induces:
$$(\mf{E}_{ell} \otimes_{\ml{E}_{ell}} M)(r) \otimes^{tr} (\mf{E}_{ell} \otimes_{\ml{E}_{ell}} N)(s) \longrightarrow (\mf{E}_{ell} \otimes_{\ml{E}_{ell}} (M \otimes_{\ml{E}_{ell}} N))(r+s).$$
And these maps are compatible with bonding maps, giving us the natural transformation:
\begin{equation}    \label{str-map} \nonumber
\rho_{M, N}: \ml{M}^{dg}(M) \otimes \ml{M}^{dg}(N) \to \ml{M}^{dg}(M \otimes N)
\end{equation}
in $\Spt^{\Sym}_{T^{tr}}(k, \mb{Q})$. Passing to homotopy categories, we obtain that  $\ml{M}$ is a lax tensor functor. 
\end{proof}

%%%%%%%%%%%%%% tensor functor
\begin{lem}   \label{tensor-fun} 
The restriction of $\ml{M}$ to $(\KCM^{GL_2, f}_{\ml{E}})^{\natural}$ is a tensor functor.
\end{lem}
\begin{proof}
By lemma \ref{laxtensor}, we only need to show that $\rho_{M, N}$ is an isomorphism in the homotopy category. Using induction on the length of the weight filtration, it's enough to show that this is an isomorphism when we take $M$ and $N$ two generalized sphere $\ml{E}_{ell}$ modules. Notice that any generalized sphere module can be realized as some idempotent of the dg module of the form $\ml{E}_{ell} \otimes \V^{a}(b)$ for some $a, b \in \mb{Z}$. We assume that $M = p(\ml{E}_{ell} \otimes \V^{a}(b))$ and $N = q(\ml{E}_{ell} \otimes \V^{c}(d))$, where $p,q$ are idempotents in the respective endo-groups. Applying Lemma \ref{fully-faithful}, we obtain that the idempotents of $\ml{E}_{ell} \otimes \V^{a}(b)$ is one-to-one corresponding to the idempotents of $M_{1}(E)^{\otimes a}(b)$, i.e., $$\ml{M}(M) = \ml{M}(p(\ml{E}_{ell} \otimes \V^{a}(b))) = \ml{M}(p)( M_{1}(E)^{\otimes a}(b)),$$
where $\ml{M}(p)$ is the image of $p$ under $\ml{M}$ in the idempotent endomorphism of $M_1(E)^{\otimes a}(b)$. Then $\rho_{M, N}$ can be identify as the morphism:
$$\ml{M}(p)(M_1(E)^{\otimes a}(b)) \otimes^{tr} \ml{M}(q)(M_1(E)^{\otimes c}(d)) \to \ml{M}(p \otimes q)(M_1(E)^{\otimes a+c}(b+d)),$$
which is an isomorphism in $\DM_{gm}(k, \mb{Q})$.
\end{proof}
%%%%%%%%% 
Before proving next lemma, we recall definitions about different kinds of generators of a triangulated category. See \cite{BB} or \cite{Rou} for example. Every subcategory $\ml{U}$ of a triangulated category $\ml{T}$ we consider is strict, which means that, each object of $ \ml{T}$, which is isomorphic to an object of $\ml{U}$, is an object of $\ml{U}$.
\begin{df} \label{generate}
Given $\mf{S}$ a set of objects in a triangulated category $\ml{T}$, then we denote $\langle \mf{S} \rangle$ to be the smallest strict full subcategory containing $\mf{S}$ and closed under finite direct sums, direct summands and shifts.
\end{df}
\begin{df} \label{operation}
Give $\ml{A}, \ml{B}$ two subcategories of a triangulated category $\ml{T}$. We define:
\begin{itemize}
\item
$\ml{A} \star \ml{B}$ is the full subcategory of $\ml{T}$ consisting of objects $X$ which can be fit into a triangle $$A \to X \to B \to A[1],$$
where $A \in \ml{A}$ and $B \in \ml{B}$.
\item
$\ml{A} \diamond \ml{B} = \langle \ml{A} \star \ml{B} \rangle$.
\item
$\langle \ml{A} \rangle_0 = 0$ and $\langle \ml{A} \rangle_n = \langle \ml{A} \rangle_{n-1} \diamond  \langle \ml{A} \rangle$ inductively.  
\item
Set $\langle \ml{A} \rangle_{\infty} = \bigcup_{n \geq 0}\langle \ml{A} \rangle_n.$
\end{itemize}
\end{df}
\begin{df} \label{classically generator}
Let $\mf{S}$ be a set of objects in a triangulated category $\ml{T}$. Then
\begin{itemize}
\item
 $\mf{S}$ classically generates $\ml{T}$ if the smallest thick (i.e. closed under isomorphisms and direct summands) subcategory of $\ml{T}$ containing $\mf{S}$ is $\ml{T}$ itself. Equivalently, $\ml{T} = \langle \mf{S} \rangle_{\infty}$.
 \item
 $\mf{S}$ generates $\ml{T}$ if, given an object $A \in \ml{T}$ such that   $$Hom_{\ml{T}}(S, A[n]) = 0$$
 for all $S \in \mf{S}$ and any $n \in \mb{Z}$, implies that $A = 0$.
\end{itemize}
\end{df}

\begin{lem} \label{cri for equivalence between  tri-cat}
Give $\ml{T}_{1}, \ml{T}_2$ two triangulated categories and $\phi: \ml{T}_{1} \to  \ml{T}_2$ a triangulated functor. Let $\mf{S}$ be a set of objects in $\ml{T}_{1}$, which classically generates $\ml{T}_{1}$ and is closed under shifts.
 Assume:
\begin{itemize}
\item[1.]
The set of the images of $\mf{S}$ under $\phi$ classically generates $\ml{T}_2$;
\item[2.]
$\phi$ restricted to $\mf{S}$, which is viewed as a full subcategory of $\ml{T}_{1}$, is fully faithful.
\end{itemize}
Then $\phi$ induces an equivalence between $\ml{T}_{1}$ and $\ml{T}_2$. 
\end{lem}
\begin{proof}
We denote the image of $\mf{S}$ by $\phi(\mf{S})$. It's enough to show that:
\

$\phi$ induces an equivalence between $\langle \mf{S} \rangle_n$ and  $\langle \phi(\mf{S})\rangle_n$ for any $n \in \mb{Z}_{\geq 0}$.
\

The case $n = 0$ is obvious. 
\

Assume $n =1$. Every object in $\langle \phi(\mf{S})\rangle$ is finite direct sums, direct summands and shifts of some objects in $\phi(\mf{S})$. Since $\phi$ is a triangulated functor, it commutes with shifts and direct sums. Because $\phi$ is fully faithful restricting on $\mf{S}$, the direct summands of an object $\phi(A)$ in $\phi(\mf{S})$ is one-to-one corresponding to the direct summands of $A \in \mf{S}$. This implies that:
$$\phi:\langle \mf{S} \rangle_1 \to \langle \phi(\mf{S})\rangle_1$$
is essential surjective. Furthermore $\phi$ is clearly fully faithful, which implies that $\phi$ is an equivalence.
\

Assume $\phi$ induces an equivalence between $\langle \mf{S} \rangle_n$ and  $\langle \phi(\mf{S})\rangle_n$. Let us prove the case $n+1$.
\

Take an element $B_{n+1}$ in $\langle \phi(\mf{S})\rangle_{n} \star \langle \phi(\mf{S})\rangle_1$, which implies that there exists a distinguished triangle:
$$B_n \to B_{n+1} \to B_1 \to B_n[1],$$
where $B_i \in \langle \phi(\mf{S})\rangle_{n}$. By induction, we know that: there exist $A_1 \in \langle \mf{S} \rangle_1$ and $A_n \in \langle \mf{S} \rangle_n$ such that:
$B_n = \phi(A_n), B_{1} = \phi(A_1)$.
\

Therefore, we have $A_{n+1} \in \langle \mf{S} \rangle_{n+1}$, such that:
$$A_n \to A_{n+1} \to A_1 \to A_n[1]$$
is a distinguished triangle in $\ml{T}_1$. Applying $\phi$ to this triangle, we get an isomorphism $\phi(A_{n+1}) \cong B_{n+1}$. After a suitable choice of the isomorphism class of $A_{n+1}$, we can find a preimage of $B_{n+1}$. 
\

In other words, we have shown that:
$$\phi: \langle \mf{S}\rangle_{n} \star \langle\mf{S}\rangle_1 \to \langle \phi(\mf{S})\rangle_{n} \star \langle \phi(\mf{S})\rangle_1$$
is essentially surjective.
\

Next, let us check that the above functor is fully faithful. Given $A, \tilde{A} \in \langle \mf{S}\rangle_{n} \star \langle\mf{S}\rangle_1$, then we can assume that there exist two distinguished triangles:
\begin{equation} \label{tri1}
A_n \to A \to A_1 \to A_n[1]
\end{equation}
and
\begin{equation} \label{tri2}
\tilde{A}_n \to \tilde{A} \to \tilde{A}_1 \to \tilde{A}_n[1].
\end{equation}
Then  applying $Hom(A_n, \cdot)$ to the triangle (\ref{tri2}), we get a long exact sequence:
$$Hom(A_n, \tilde{A}_n) \to Hom(A_n, \tilde{A}) \to Hom(A_n, \tilde{A}_1) \to Hom(A_n, \tilde{A}_n[1]) \to \cdots$$
After compared to the image of the above long exact sequence under $\phi$, and by induction on $n$ and the five lemma, we get that:
$$Hom(A_n, \tilde{A}[*]) \cong Hom(\phi(A_n), \phi( \tilde{A})[*]).$$
Similarly, we have 
$Hom(A_1, \tilde{A}[*]) \cong Hom(\phi(A_1), \phi( \tilde{A})[*]).$
\

Next applying $Hom( \cdot, \tilde{A})$ to the triangle (\ref{tri2}), we get another long exact sequence:
$$Hom(A_n[1], \tilde{A}) \to Hom(A_1, \tilde{A}) \to Hom(A, \tilde{A}) \to Hom(A_n, \tilde{A}) \to \cdots.$$
Compared to its image under $\phi$ and isomorphisms above, we get:
$$Hom(A, \tilde{A}[*]) \cong Hom(\phi(A), \phi( \tilde{A})[*]).$$
Now, we have shown that:
$$\phi: \langle \mf{S}\rangle_{n} \star \langle\mf{S}\rangle_1 \to \langle \phi(\mf{S})\rangle_{n} \star \langle \phi(\mf{S})\rangle_1$$
is an equivalence. 
\

Recall that $\phi$ commutes with shifts and finite direct sums, and maps the idempotent in $End(A)$ to the idempotent in $End(\phi(A))$ for any $A \in \langle \mf{S}\rangle_{n} \star \langle\mf{S}\rangle_1$. This implies that:
$$\phi: \langle \mf{S}\rangle_{n} \diamond \langle\mf{S}\rangle_1 \to \langle \phi(\mf{S})\rangle_{n} \diamond \langle \phi(\mf{S})\rangle_1$$
is an equivalence.
\end{proof}
\begin{thm} \label{equi}
Given $E$ an elliptic curve without CM, then there is an exact functor
$$\ml{M}: \ml{D}^{GL_2}_{\ml{E}_{ell}} \to \DM(k, \mb{Q}),$$
which is a lax tensor functor. Furthermore, the restriction of $\ml{M}$ to
$$\ml{M}^c: (\ml{D}^{GL_2}_{\ml{E}_{ell}})^c \to \DM(k, \mb{Q})$$
defines an equivalence of $(\ml{D}^{GL_2}_{\ml{E}_{ell}})^c$ with $\DMEM(k, \mb{Q})_E$ as triangulated tensor categories, where $(\ml{D}^{GL_2}_{\ml{E}_{ell}})^c$ is the full subcategory of $\ml{D}^{GL_2}_{\ml{E}_{ell}}$ consisting of compact objects.
\end{thm}
\begin{proof}
By Lemma \ref{tensor-fun} and Lemma \ref{comp-Tate}, we know that the restriction of $\ml{M}$ to $(\ml{D}^{GL_2}_{\ml{E}_{ell}})^c$ is a tensor functor with $\ml{M}(\ml{E}_{ell} \otimes \V^{a}(b))  \cong M_1(E)^{\otimes a}(b)$.
\

From Lemma \ref{fully-faithful}, we have:
\begin{equation} \nonumber
\begin{split}
& Hom_{\KCM^{GL_2}_{\ml{E}_{ell}}}(\ml{E}_{ell} \otimes \V^{\otimes a}(b), \ml{E}_{ell}  \otimes \V^{\otimes c}(d)[i]) 
\cong  Hom_{\DM_{gm}(k, \mb{Q})}(M_{1}(E)^{\otimes a}(b), M_{1}(E)^{\otimes c}(d)[i]).
\end{split}
\end{equation}
One can check this isomorphim is induced by the functor $\ml{M}$. Using Lemma \ref{cri for equivalence between  tri-cat} and Corollary \ref{classically generates}, we obtain that $\ml{M}^{c}$ gives an equivalence between $(\ml{D}^{GL_2}_{\ml{E}_{ell}})^c$ and $\DMEM(k, \mb{Q})_E$.
\end{proof}
\begin{rmk} \label{CM case}
Assume that $E$ is an elliptic curve with CM. For $r \in \mb{Z}_{\geq 0}$ and $M \in \ml{CM}^{T_{\mb{K}}}_{\ml{E}^*_{ell}}$, we define its Adams graded $r$ summand is defined as:
$$M(r) = Hom_{T_{\mb{K}}}(det^{\otimes -r}, \mf{E}_{ell}^{*} \otimes_{\ml{E}_{ell}^{*}}  M[2r]).$$
Then we can repeat all the above constructions and get the same results. In particular, we have the following theorem.
\begin{thm} \label{equicm}
Given $E$ an elliptic curve with CM, then there is an exact functor
$$\ml{M}: \ml{D}^{T_{\mb{K}}}_{\ml{E}_{ell}} \to \DM(k, \mb{K}),$$
which is a lax tensor functor. Furthermore, the restriction of $\ml{M}$ to
$$\ml{M}^c: (\ml{D}^{T_{\mb{K}}}_{\ml{E}_{ell}})^c \to \DM(k, \mb{K})$$
defines an equivalence of $(\ml{D}^{T_{\mb{K}}}_{\ml{E}_{ell}})^c$ with $\DMEM(k, \mb{K})_E$ as triangulated tensor categories, where $(\ml{D}^{{T_{\mb{K}}}}_{\ml{E}_{ell}})^c$ is the full subcategory of $\ml{D}^{T_{\mb{K}}}_{\ml{E}_{ell}}$ consisting of compact objects.
\end{thm}
\end{rmk}

\begin{conj} \label{BS vanishing conjecture}
(\textbf{The generalized Beilison-Soul\'{e} vanishing conjecture for an elliptic curve}) 
\begin{enumerate}
\item
An elliptic curve $E$ over a field $k$ without CM satisfies the conditions:
$$Hom_{\DM_{gm}(k, \mb{Q})}(M_1(E)^{\otimes a}, \mb{Q}(a-b)[m]) = 0$$
in the following two cases:
\begin{itemize}
\item[A.]
$a = 0, b < 0, m \leq 0$;
\item[B.]
$a > 0, a \geq 2b, m \leq 0$.
\end{itemize}
\item
An elliptic curve $E$ over a field $k$ with CM, whose $1$-motive $M_1(E) = M \oplus \bar{M}$, satisfies the conditions:
$$Hom_{\DM_{gm}(k, \mb{K})}(M^{\otimes a}, \mb{K}(a-b)[m]) = 0$$
and
$$Hom_{\DM_{gm}(k, \mb{K})}(\bar{M}^{\otimes a}, \mb{K}(a-b)[m]) = 0$$
in the following two cases:
\begin{itemize}
\item[A.]
$a = 0, b < 0, m \leq 0$;
\item[B.]
$a > 0, a \geq 2b, m \leq 0$.
\end{itemize}
\end{enumerate}

\end{conj}
\begin{rmk} \label{relation with the B-S}
In fact, Part (A) of Conjecture \ref{BS vanishing conjecture} is the classical Beilison-Soul\'{e} vanishing conjecture. See \cite{L2} for example. In fact, all of these generalized conjectures can be expressed as follows. The strong Beilison-Soul\'{e} vanishing conjecture for $X$:
\

($BS^{*}_X$) For any smooth $k$-scheme $X$, $H^n(X, \mb{Q}(i)) = 0$ provided $n \leq 0$ and $i > 0$.
\

\noindent When $X$ is a field, the conjectures $BS^*_{X}$ is called the strong Beilison-Soul\'{e} vanishing conjecture in \cite{L}. Conjecture \ref{BS vanishing conjecture} is the same as $BS^*_{E^n}$ for $n \in \mb{Z}_{\geq 0}$.
\end{rmk}

\begin{cor} \label{abelian category of elliptic motives}
Assume that $E$ is an elliptic curve without CM, satisfies the $r$-th vanishing properties for $r \geq 0$ and the generalized Beilison-Soul\'{e} vanishing conjecture, then:
\begin{itemize}
\item[1.]
$\DMEM(k, \mb{Q})_E$ has a t-structure which is induced from $$\ml{M}^{f}: \ml{D}^{GL_2, f}_{\ml{E}} \to \DMEM(k, \mb{Q})_{E},$$
where $\ml{M}^{f}$ is the restriction of the functor $\ml{M}$ (Theorem \ref{equi}) to $\ml{D}^{GL_2, f}_{\ml{E}}$. Denote its heart by $\MEM(k, \mb{Q})_E$.
\item[2.]
$\ml{M}^{f}$ induces an equivalence of Tannakian categories:
$$H^0(\ml{M}^{f}): \ml{H}^{GL_2, f}_{\ml{E}} \to \MEM(k, \mb{Q})_E.$$
\end{itemize}
\end{cor}
\begin{proof}
First, it follows from our assumptions and Theorem \ref{equi} that $\ml{E}_{ell} \cong \ml{E}$ is a cohomologically connected cdga over $GL_2$. Then by Theorem 8.4 in \cite{C}, we have a $t$-structure on $\ml{D}^{GL_2, f}_{\ml{E}}$. Therefore the equivalence of Theorem \ref{equi} gives us an induced $t$-structure on $\DMEM(k, \mb{Q})_{E}$, which satisfies the desired properties. 
\end{proof}
\begin{cor} \label{abelian category of elliptic motives cm}
Assume that $E$ is an elliptic curve with CM, satisfies the $r$-th vanishing properties for $r \geq 0$ and the generalized Beilison-Soul\'{e} vanishing conjecture, then:
\begin{itemize}
\item[1.]
$\DMEM(k, \mb{K})_E$ has a t-structure which is induced from $$\ml{M}^{f}: \ml{D}^{T_{\mb{K}}, f}_{\ml{E}} \to \DMEM(k, \mb{K})_{E},$$
where $\ml{M}^{f}$ is the restriction of the functor $\ml{M}$ (Theorem \ref{equi}) to $\ml{D}^{T_{\mb{K}}, f}_{\ml{E}}$. Denote its heart by $\MEM(k, \mb{K})_E$.
\item[2.]
$\ml{M}^{f}$ induces an equivalence of Tannakian categories:
$$H^0(\ml{M}^{f}): \ml{H}^{T_{\mb{K}}, f}_{\ml{E}} \to \MEM(k, \mb{K})_E.$$
\end{itemize}
\end{cor}

\section{Relation with mixed Tate motives} \label{relation with MT}
In this section, we put the constructions of the Adams cycle algebra for mixed Tate motives into our setting. As before, we only work out the case of elliptic curves without CM in detail. In the CM case, the construction is similar. Firstly we recall the definitions in Chapter 4 of \cite{L} .
\

\begin{df}
We let $\Ztr((\mb{P}^1/\infty)^q))$ be defined by the cokernel of the map:
\[
\oplus_{j = 1}^r \Ztr((\mb{P}^1)^{q - 1})) \xrightarrow{\sum_j i_{j, \infty *}} \Ztr((\mb{P}^1)^q))
\]
where $i_{j, \infty}: (\mb{P}^1)^{q-1} \to (\mb{P}^1)^q$ inserts $\infty$ in the $j$-th place.
\end{df}
\begin{df}
The Adams cycle algebra for mixed Tate motives is defined by:
\begin{equation} \nonumber
\ml{N} = \mb{Q} \oplus \bigoplus_{q \geq 1}\ml{N}(q), 
\end{equation}
where $\ml{N}(q) \subset C^{Alt}_{*}(\Ztr((\mb{P}^1/\infty)^q))$ be the subsheaf of symmetric sections with respect to the action of symmetric group $\Sigma_q$ by permuting the coordinates in $(\mb{P}^1)^q$.
\end{df}
\begin{rmk}
One can show that the homotopy category of finite cell $\ml{N}$-modules can be identified as the triangulated category of mixed Tate motives $\mathbf{DMT}(k, \mb{Q})$, which is a full rigid tensor subcategory of $\DM_{gm}(k, \mb{Q})$ generated by Tate objects. The proof can be found in Section 5.3 in \cite{L}. In fact, one of the main results in \cite{L} is to show this equivalence can be generalized to mixed Tate motives over a base scheme that is separated, smooth and essentially of finite type over a field. Along with the strategy in \cite{L}, we also want to generalize our results into mixed elliptic motives over a general base scheme in the future.
\end{rmk}

\begin{df}
We define the modified Adams cycle algebra for mixed Tate motives by:
\begin{equation} \nonumber
\widehat{\ml{N}} = \mb{Q} \oplus \bigoplus_{t \geq 1, t \in \mb{Z}}\widehat{\ml{N}}_{2t}, 
\end{equation}
where $\widehat{\ml{N}}_{2t} = \ml{N}(t) \otimes det^{\otimes -t}$.
\end{df}
\begin{rmk} \label{N is sub-alg of E}
By Definition \ref{E^{a,b}}, we know that: 
$\ml{E}^*_{0, b} = \ml{N}(-b) \otimes det^{\otimes b}$
for any $b \in \mb{Z}_{\leq 0}$. This implies that $\widehat{\ml{N}}_{2t} \subset \ml{E}_{2t}$. Using the algebra structure of $\ml{N}$ (Section 4.2 in \cite{L}) and the tensor structure of determinant representations (viewed as $GL_2$ representations), we know that $\widehat{\ml{N}}$ is sub-algebra of $\ml{E}_{ell}$ as a cdga over $GL_2$.
\end{rmk}
\begin{rmk}
Notice that our Adams grading is different from Adams grading defined in \cite{L}. More precisely, Adams degree $r$ in the sense of \cite{L} is Adams degree $2r$ in our sense.
%There is a natural isomorphism between the category of Adams graded dg $\ml{N}$ module and Adams graded dg $\widehat{\ml{N}}$ module. 
\end{rmk}
We define $\ml{CM}^{\mb{G}_m}_{A}$ to be the category of cell modules of Tate-type for a cdga $A$ over $GL_2$, i.e. cell modules consisting only by the generalized sphere modules of the form $A[-n] \otimes det^{\otimes r}$, which is a full subcategory of $\ml{CM}^{GL_2}_{A}$.
\begin{rmk}  \label{Psi_1}
There is a natural functor: $$\Psi_1: \ml{CM}_{\ml{N}} \to \ml{CM}^{\mb{G}_m}_{\widehat{\ml{N}}},$$
which sends the cell module $\ml{N}\langle n \rangle$, defined in Example 1.4.5 of \cite{L}, to the cell module $\widehat{\ml{N}} \otimes det^{\otimes n}$. $\Psi_1$ induces a functor between their associated homotopy categories, even homotopy categories of finite cell modules. For simplicity, we denote both of these functors by $\Psi_1$. In particular, we have:
$$\Psi_1: \ml{D}^{f}_{\ml{N}} \to \ml{D}^{\mb{G}_m, f}_{\widehat{\ml{N}}}.$$
\end{rmk}
Notice that the inclusion: $\ml{CM}^{\mb{G}_m}_{\widehat{\ml{N}}} \to \ml{CM}^{GL_2}_{\widehat{\ml{N}}}$ induces a functor $$\Psi_2: \ml{D}^{\mb{G}_m}_{\widehat{\ml{N}}} \to \ml{D}^{GL_2}_{\widehat{\ml{N}}}.$$
Similarly, on the level of homotopy category of finite cell modules, we have:
 $$\Psi_2: \ml{D}^{\mb{G}_m, f}_{\widehat{\ml{N}}} \to \ml{D}^{GL_2, f}_{\widehat{\ml{N}}}.$$
\begin{rmk}
Because $\widehat{\ml{N}}$ is Adams connected, we have:
$$\ml{D}^{GL_2, f}_{\widehat{\ml{N}}} \cong (\ml{D}^{GL_2}_{\widehat{\ml{N}}})^c.$$ 
\end{rmk}
Using Remark \ref{N is sub-alg of E}, we have a map between cdgas over $GL_2$:
$\widehat{N} \xrightarrow{i} \ml{E}_{ell}.$
This induces a functor:
$$\Psi_3: \ml{CM}^{GL_2}_{\widehat{\ml{N}}} \to \ml{CM}^{GL_2}_{ \ml{E}_{ell}},$$
which sends $M$ to $M \otimes_{\widehat{\ml{N}}} \ml{E}_{ell}$. Furthermore, we have:
$$\Psi_3: \ml{D}^{GL_2}_{\widehat{\ml{N}}} \to \ml{D}^{GL_2}_{ \ml{E}_{ell}}$$
and
$$\Psi_3: (\ml{D}^{GL_2}_{\widehat{\ml{N}}})^c \to (\ml{D}^{GL_2}_{ \ml{E}_{ell}})^c.$$
From our constructions of $\Psi_i, i = 1, 2, 3$, we have the following statement.
\begin{prop}
We have the following commutative diagram:
\begin{center}
\begin{xy}
(180,0)*+{\ml{D}^{f}_{\ml{N}}}="v1";
(200,0)*+{\ml{D}^{\mb{G}_m, f}_{\widehat{\ml{N}}}}="v2";
(220,0)*+{\ml{D}^{GL_2, f}_{\widehat{\ml{N}}}}="v3";(240,0)*+{(\ml{D}^{GL_2}_{\widehat{\ml{N}}})^c}="v4";
(270,0)*+{(\ml{D}^{GL_2}_{ \ml{E}_{ell}})^c}="v5";
(180,-20)*+{\mathbf{DMT}(k, \mb{Q})}="v6";
(270,-20)*+{\DMEM(k, \mb{Q})_E}="v7";
{\ar@{->}^{\Psi_1} "v1";"v2"};{\ar@{->}^{\Psi_2} "v2";"v3"};
{\ar@{->}^{\cong} "v3";"v4"};{\ar@{->}^{\Psi_3} "v4";"v5"};
{\ar@{->}^{\ml{M}} "v1";"v6"};{\ar@{->}^{\ml{M}^c} "v5";"v7"};
{\ar@{->} "v6";"v7"};
\end{xy}
\end{center}
\noindent where the left vertical map $\ml{M}$ is defined in Section 5.3 of \cite{L} and the right vertical map $\ml{M}^c$ is defined in Section \ref{DG modules and motives for an elliptic curve}. In particular, the composition of top arrows is fully faithful.
\end{prop}

\bigskip

%%========================================================%%%
%%%        Affiliation
%%========================================================%%%
Jin Cao, Yau Mathematical Sciences Center, Tsinghua University, Beijing, China,
\

\textit{E-mail address}: jcao@math.tsinghua.edu.cn


\begin{thebibliography}{99}
\bibitem{A}
{\sc G. Ancona}: \newblock D\'{e}composition du motif d'un sch\'{e}ma ab\'{e}lien universal. Ph.D thesis. 2012.

\bibitem{AEH}
{\sc G. Ancona, S. Enright-Ward, and A. Huber}: \newblock On the motive of a commutative algebraic group. 2016. \newblock arXiv: 1312.4171v2. 

\bibitem{An}
{\sc Y. Andr\'{e}}: \newblock Une introduction aux motifs: Motifs purs, motifs mixtes, p\'{e}riodes. \newblock SMF.  Panoramas et Synth\`{e}ses. No. 17. 2004.
 
%\bibitem{BS}
%{\sc P. Balmer, M. Schlichting}: \newblock Idempotent completion of triangulated categories. \newblock J. Algebra 236. No. 2, 2001.

%\bibitem{Bei}
%{\sc A. Beilinson}: \newblock Height pairing between algebraic cycles. In K-Theory, Arithmetic and Geometry. \newblock LNM1289. 1987.

\bibitem{B}
{\sc S. Bloch}: \newblock Algebraic cycles and the Lie algebra of mixed Tate motives. \newblock J. AMS 4, No. 4. 1991.

\bibitem{BK}
{\sc S. Bloch, I. Kriz}: \newblock Mixed Tate motives. \newblock Ann of Math. 140, no.3. 1994. 557-605.

\bibitem{BB}
{\sc A. Bondal, M. Van den Bergh}: \newblock Generators and representability of functors in commutative and noncommutative geometry. \newblock Moscow Math. J. 3.  2003. 1-36.

%\bibitem{Bor}
%{\sc A. Borel}: \newblock Stable real cohomology of arithmetic groups. \newblock Ann. Sci. \'{E}c. Norm. Sup. 7. 1974. 235-272.

%\bibitem{BG}
%{\sc A. Bousfield, V. Guggenheim}: \newblock On PL de Rham theory and rational homotopy type. \newblock Memoirs of the A.M.S. 179. 1976.

\bibitem{C}
{\sc J. Cao}: \newblock Differential graded algebras over some reductive group. 2017. \newblock arXiv: 1704.01734.

\bibitem{CD}
{\sc D. C. Cisinski, F. D\'{e}glise}: \newblock Triangulated categories of mixed motives. 2009. \newblock arXiv: 0912.21110v3.

\bibitem{Del02}
{\sc P. Deligne}: \newblock Cat$\acute{e}$gories tensorielles. \newblock Mosc. Math. J. 2.  2002. 227-248.

%\bibitem{DM}
%{\sc P.Deligne, J. S. Milne}: \newblock Tannakian Categories. in Hodge cycles, Motives, and Shimura Varieties. \newblock LNM 900. 1982. 101-228. 

%\bibitem{FHT}
%{\sc Y. F$\acute{e}$lix, S. Halperin, and J. C. Thomas}: \newblock Rational homtopy theory. \newblock GTM 205. Springer. 2001.

%\bibitem{FH}
%{\sc W. Fulton, J. Harris}: \newblock Representation theory. A first course. \newblock GTM 129. Springer. 1999.

\bibitem{FSV}
{\sc E. Friedlander, A. Suslin and V.Voevodsky}: \newblock Cycles, transfer and motivic homology theories. \newblock Annals of Math.Studies 143, Princeton Univ. Press. 2000.

%\bibitem{HM1}
%{\sc R. Hain, M. Matsumoto}: \newblock Tannakian Fundamental Groups Associated to Galois Groups, Galois groups and Fundamental Groups. \newblock MSRI Publications, vol. 41.  Cambridge University Press. 2003. 183-216.

%\bibitem{HM2}
%{\sc R. Hain,  M. Matsumoto}: \newblock Universal mixed elliptic motives. 2015. \newblock arXiv:1512.03975. 

%\bibitem{J}
%{\sc J. C. Jantzen}: \newblock Representations of algebraic group. \newblock Acad. Press. 1987.
\bibitem{I1}
{\sc I. Iwanari}: \newblock Tannaka duality and stable infinity-categories. 2015. \newblock arXiv:1409.3321v2.

\bibitem{I2}
{\sc I. Iwanari}: \newblock Mixed motives and quotient stacks. 2016. \newblock arXiv:1307.3175v6.

\bibitem{KT}
{\sc K. Kimura, T. Terasoma}: \newblock Relative DGA and mixed elliptic motives.  2010. \newblock arXiv:1010.3791v3.

\bibitem{Kim}
{\sc S.-I. Kimura}: \newblock Chow groups are finite-dimensional, in some sense. \newblock Math. Ann. 331, 2005. 173-201.  

\bibitem{KM}
{\sc I. Kriz, J. P. May}: \newblock Operads, algebras, modules and motives. \newblock Ast$\acute{e}$risque. No. 233. 1995.

\bibitem{L2}
{\sc M. Levine}: \newblock Tate motives and the vanishing conjectures for algebraic K-theory. \newblock Algbebraic K-Theory and Algebraic Topology. NATO ASI Series C. Vol. 407. 1993. 167-188.

\bibitem{L3}
{\sc M. Levine}: \newblock Mixed Motives. \newblock Handbook of K-theory, Vol 1, Springer Verlag, 2005. 429-522. 

\bibitem{L}
{\sc M. Levine}: \newblock Tate motives and the fundamental group. \newblock Cycles, Motives and Shimura Varieties. Tata Institute of Fundamental Research, Mumbai, India, August 2010. 265-392. 

\bibitem{LMS}
{\sc L. G. Lewis, Jr., J. P. May, and M. Steinberger (with contributions by J. E. McClure)}: \newblock Equivariant stable homotopy theory. \newblock Springer Lecture Notes. Vol 1213. 1986.

%\bibitem{ML}
%{\sc S. Mac Lane}: \newblock Categories for the working mathematician, second ed., \newblock Graduate Texts in Mathematics, vol. 5, Springer-Verlag, 1998.

\bibitem{Ma}
{\sc C. Mazza}: Schur functors and motives. Ph.D thesis. 2004. 

\bibitem{MVW}
{\sc C. Mazza, V. Voevodsky and C. Weibel}: \newblock Lecture notes on motivic cohomology. \newblock American Mathematical Society. 2006.

%\bibitem{Mi}
%{\sc J.S. Milne}: \newblock Algebraic Groups: algebraic group schemes over fields. Course note. 2015.

\bibitem{Nee}
{\sc A. Neeman}: \newblock The connection between the K-theory localization theorem of Thomason, Trobaugh and Yao and the smashing subcategories of Bousfield and Ravenel.  \newblock Ann. Sci. $\acute{E}$cole. Norm. Sup. (4) 25  no. 5.1992. 547-566. 

\bibitem{P}
{\sc O. Patashnick}: \newblock A Candidate for the abelian category of mixed elliptic motives. \newblock Journal of K-Theory. 12(3). 2013. 569-600.


\bibitem{Rou}
{\sc R. Rouquier}: \newblock Dimensions of triangulated categories. \newblock Journal of K Theory, 1. 2008.1-36

\bibitem{S}
{\sc M. Spitzweck}: \newblock Operads, algebras and modules in model categories and motives. Ph.D thesis. 2001.

%\bibitem{Sul}
%{\sc D. Sullivan}: \newblock Infinitesimal computations in topology. \newblock {\em Publications Math\'ematiques, Institut des Hautes \'Etudes Scientifiques}, 47.1978. 269-331.

\bibitem{V10}
{\sc V. Voevodsky}: Cancellation theorem. \newblock Doc. Math. (Extra volume: Andrei A. Suslin sixtieth birthday).  2010. 671-685.

%\bibitem{Wei}
%{\sc C. A. Weibel}: \newblock An introduction to homological algebra. \newblock Cambridge Studies in Advanced Mathematics 38. 1995.





%\bibitem{IntersectionFulton}
%{\sc W.~Fulton}, {\em Intersection theory}, vol.~2 of Ergebnisse der
%Mathematik und ihrer Grenzgebiete. 3.
%  Folge. A Series of Modern Surveys in Mathematics [Results in Mathematics and
%  Related Areas. 3rd Series. A Series of Modern Surveys in Mathematics],
%  Springer-Verlag, Berlin, second~ed., 1998.



\end{thebibliography}
\end{document}